\documentclass[a4paper]{article}

\usepackage{tikz}

\usepackage{tikz-cd}
\usetikzlibrary{matrix,arrows, patterns, positioning, calc, intersections}

\usepackage[utf8]{inputenc}

\usepackage{microtype}
\usepackage{amscd}
\usepackage{amsmath,amssymb,amsfonts, accents}
\usepackage{xstring}
\usepackage{xcolor}
\usepackage[mathscr,mathcal]{euscript}
\usepackage{xypic}
\xyoption{rotate}
\usepackage[cmtip,all,2cell]{xy}
\entrymodifiers={+!!<0pt,\fontdimen22\textfont2>}
\usepackage{url}
\usepackage{latexsym}
\usepackage{amsthm}
\usepackage{multicol}
\usepackage{enumitem}
\usepackage[colorlinks=true, linkcolor={blue!50!black}, pdfhighlight=/O, ocgcolorlinks=true]{hyperref}
\usepackage{graphicx}
\usepackage{color}
\usepackage{mathtools}
\usepackage[colorinlistoftodos]{todonotes}
\usepackage[affil-it]{authblk}

\usepackage{adjustbox}

\usepackage{cleveref}

\usepackage[a4paper, total={5.5in,8.5in}]{geometry}

\numberwithin{equation}{section}
\theoremstyle{plain}
\newtheorem{theorem}[equation]{Theorem}
\newtheorem{lemma}[equation]{Lemma}
\newtheorem{proposition}[equation]{Proposition}
\newtheorem{corollary}[equation]{Corollary}

\theoremstyle{definition}

\newtheorem{definition}[equation]{Definition}

\newtheorem{example}[equation]{Example}
\newtheorem{examples}[equation]{Examples}
\newtheorem{remark}[equation]{Remark}

\setlist[enumerate]{label=(\arabic*), leftmargin=*}
\setenumerate{label=(\arabic*), leftmargin=*}
\setlist[itemize]{label=$\vcenter{\hbox{\footnotesize$\bullet$}}$, leftmargin=*}
\setitemize{label=$\vcenter{\hbox{\footnotesize$\bullet$}}$, leftmargin=*}

\newcommand{\mb}[1]{\mathbf{#1}}
\newcommand{\mm}[1]{\mathrm{#1}}

\newcommand{\ul}[1]{\underline{#1}}

\newcommand{\id}{ \mathrm{id}}

\newcommand{\cat}[1]{
\StrLen{#1}[\mystrlen]
\ifnum\mystrlen=1 \mathscr{#1}
\else \mathrm{#1}
\fi}
\newcommand{\scat}[1]{\mb{#1}}

\newcommand{\Hom}[0]{\mm{Hom}}

\newcommand{\op}[0]{\mm{op}}

\newcommand{\DR}{\mathrm{DR}}

\newcommand{\cA}{\mathcal{A}}
\newcommand{\cC}{\mathcal{C}}

\makeatletter
\newcommand{\mytag}[2]{%
  \text{#1}%
  \@bsphack
  \begingroup
    \@onelevel@sanitize\@currentlabelname
    \edef\@currentlabelname{%
      \expandafter\strip@period\@currentlabelname\relax.\relax\@@@%
    }%
    \protected@write\@auxout{}{%
      \string\newlabel{#2}{%
        {#1}%
        {\thepage}%
        {\@currentlabelname}%
        {\@currentHref}{}%
      }%
    }%
  \endgroup
  \@esphack
}
\makeatother

\title{Calabi--Yau structures on (quasi-)bisymplectic algebras}

\author[]{
	Tristan Bozec\thanks{IMAG, Univ. Montpellier, CNRS, Montpellier, France \\
											\href{mailto:tristan.bozec@umontpellier.fr}{tristan.bozec@umontpellier.fr}}, 
	Damien Calaque\thanks{IMAG, Univ. Montpellier, CNRS, Montpellier, France \\
											\href{mailto:damien.calaque@umontpellier.fr}{damien.calaque@umontpellier.fr}}, 
	Sarah Scherotzke\thanks{Mathematical Institute, University of Luxembourg, Luxembourg \\
											\href{mailto:sarah.scherotzke@uni.lu}{sarah.scherotzke@uni.lu}}}

\date{}

\begin{document}

\maketitle

\begin{abstract}
We show that relative Calabi--Yau structures on noncommutative moment maps give rise to (quasi-)bisymplectic structures, as introduced by 
Crawley-Boevey--Etingof--Ginzburg (in the additive case) and Van den Bergh (in the multiplicative case). 
We prove along the way that the fusion process (a) corresponds to the composition of Calabi--Yau cospans with ``pair-of-pants'' ones, 
and (b) preserves the duality between non-degenerate double quasi-Poisson structures and quasi-bisymplectic structures. 

As an application we obtain that Van den Bergh's Poisson structures on the moduli spaces of
representations of deformed multiplicative preprojective algebras coincide with 
the ones induced by the $2$-Calabi--Yau structures on (dg-versions of) these algebras. 
\end{abstract}

\setcounter{tocdepth}{2}
\tableofcontents


\section{Introduction}

Throughout this paper $k$ is a field of characteristic zero. 

\subsubsection*{Noncommutative algebraic geometry}

The Kontsevich--Rosenberg principle of noncommutative algebraic geometry says that a structure on an associative algebra $A$ has a (noncommutative) geometric 
meaning whenever it induces a genuine corresponding geometric structure on representation spaces. This principle led to the discovery of bisymplectic 
structures~\cite{CBEG}, double Poisson and double quasi-Poisson structures~\cite{VdB}, and quasi-bisympletic structures~\cite{VdB2} on smooth algebras such that the 
associated representation spaces are respectively hamiltonian $GL_n$-varieties, Poisson and quasi-Poisson $GL_n$-varieties, and quasi-hamiltonian 
$GL_n$-varieties. 

It turns out that the fusion procedure for (quasi-)hamiltonian spaces from~\cite{AKSM,AMM} has a noncommutative counterpart~\cite{VdB,VdB2} (also called fusion). 
This in particular allows to construct quasi-bisymplectic structures on (localisations of) path algebras of quivers by starting from several copies of $A_2$ and repeatedly 
applying the fusion procedure. Ultimately, this provides a construction of symplectic structures~\cite{Yama} on multiplicative quiver varieties~\cite{CBS}. 
\[
\begin{array}{|c|c|}
\hline
\mathbf{Noncommutative~algebra} & \mathbf{Algebraic~geometry}\\
\hline \hline
\textrm{Smooth~algebra~}A & \textrm{Representation~variety}~\mathbf{Rep}(A)\\ 
\hline
\textrm{Bisymplectic~algebras} & \textrm{Hamiltonian~}GL\textrm{-spaces}\\ 
\hline
\textrm{Quasi-bisymplectic~algebras} & \textrm{Quasi-hamiltonian~}GL\textrm{-spaces}\\ 
\hline
\textrm{Fusion} & \textrm{Fusion}\\ 
\hline
\end{array}
\]

\subsubsection*{Derived symplectic geometry}

Hamiltonian and quasi-hamiltonian spaces actually find a nice interpretation (see~\cite{Cal,Saf})
in the realm of shifted symplectic and lagrangian structures from~\cite{PTVV}: moment maps as well as their multiplicative analogs 
naturally lead to lagrangian morphisms, and both the reduction and the fusion procedures can be understood in terms of derived intersections of these. 
\[
\begin{array}{|c|c|}
\hline
\mathbf{Algebraic~geometry} & \mathbf{Derived~geometry} \\
\hline \hline
G\ \rotatebox[origin=c]{-90}{$\circlearrowright$}\ X & \textrm{Quotient stack}~[X/G]\\ 
\hline
\textrm{Hamiltonian~}G\textrm{-space~}X & \textrm{Lagrangian~morphism}~[X/G]\to [\mathfrak{g}^*/G]\\ 
\hline
\textrm{Quasi-hamiltonian~}G\textrm{-space~}X & \textrm{Lagrangian~morphism}~[X/G]\to [G/G]\\ 
\hline
\textrm{Reduction} & \textrm{Lagrangian~intersection}\\ 
\hline
\textrm{Fusion} & \textrm{Composing~Lagrangian~correspondences}\\ 
\hline
\end{array}
\]

\subsubsection*{Calabi--Yau structures}

More recently, absolute and relative Calabi--Yau structures~\cite{BD1} have turned out to be accurate noncommutative analogs of shifted symplectic and lagrangian 
structures~\cite{BD2,ToCY}, \textit{via} the moduli of object functor $\mathbf{Perf}$ from~\cite{ToVa}. 
\[
\begin{array}{|c|c|}
\hline
\textbf{Higher~algebra} & \mathbf{Derived~geometry} \\
\hline \hline
\textrm{Finite~type~dg-category~}\mathcal C & \textrm{Derived~Artin~stack~}\mathbf{Perf}(\mathcal C)\\ 
\hline
\textrm{Shifted~Calabi--Yau~structure} & \textrm{Shifted~symplectic structure}\\ 
\hline
\textrm{Relative~Calabi--Yau~structure} & \textrm{Lagrangian~structure}\\ 
\hline
\textrm{Calabi--Yau~pushout} & \textrm{Lagrangian~intersection}\\ 
\hline
\textrm{Composing~Calabi--Yau~cospans} & \textrm{Composing~Lagrangian~correspondences}\\ 
\hline
\end{array}
\]
It is therefore natural to wonder whether Calabi--Yau structures are hidden behind the (quasi-)bisymplectic ones aforementioned. 
More specifically, in our previous work~\cite{BCS,BCS2}, we constructed relative Calabi--Yau structures on (multiplicative) noncommutative moment maps 
$k[x^{(\pm1)}]\to A$ for (multiplicative) preprojective algebras associated with quivers, leading in particular to an alternative construction 
of symplectic structures on multiplicative quiver varities. Exhibiting a direct connection between Calabi--Yau and (quasi-)bisymplectic structures will then help 
identifying the induced symplectic structures on multiplicative quiver varities from both approaches. 

\subsubsection*{Results}

In a very satisfactory manner, relative Calabi--Yau structures on noncommutative moment maps do induce (quasi-)bisymplectic ones: the additive version is proved 
by our first main result~\cref{rcybisym}, the multiplicative one is given by~\cref{thmcyqbs}. 
The rough idea in each case is that the Calabi--Yau structure on $k[x^{(\pm1)}]\to A$ is given by a family of noncommutative forms 
$\omega_n\in\Omega^{2n}A$, $n\ge1$, satisfying conditions implying the required ones for the $2$-form $\omega_1$ to define a (quasi-)bisymplectic structure 
on $A$. In particular, non-degeneracy on the Calabi--Yau side implies non-degeneracy on the (quasi-)bisymplectic side.

We moreover prove that we retrieve for quivers the very same structures exhibited in~\cite{CBEG,VdB}, in~\cref{exqcbeg} in the additive case, and in a way more involved way in~\cref{subsecexqq} in the multiplicative case. This requires to work on the elementary $A_2$ quiver as well as on the correct realization of fusion in the framework of Calabi--Yau cospans. We need for the latter to prove in~\cref{section: fusion} (along with~\cref{fusthmadd} and~\cref{1cyqbs}) that fusion actually corresponds to composition of relative Calabi--Yau structures with a particular Calabi--Yau cospan studied in~\cite{BCS2}, the ``pair-of-pants'' one, that is\[
k[x^{(\pm1)}] \amalg k[y^{(\pm1)}]  \longrightarrow  k\langle x^{(\pm1)} ,y^{{(\pm1)}} \rangle\longleftarrow k[z^{(\pm1)}]\]
where $z$ is mapped to $x+y$ in the additive version, $xy$ in the multiplicative one.
\[
\begin{array}{|c|c|}
\hline
\textbf{Higher~algebra} & \mathbf{Noncommutative~algebra} \\
\hline \hline
\textrm{Finite~linear~category~}\mathcal C & \textrm{Path~algebra~}A_{\mathcal C}\\ 
\hline
\textrm{Object~}i & \textrm{Primitive~idempotent~}e_i \\
\hline
\textrm{Calabi--Yau~functor~}\coprod_{i}k[x_i]\to\mathcal C & 
\begin{matrix}\textrm{Bisymplectic~structure,} \\ \textrm{with~moment~map~}k[x]\to A_{\mathcal C}\end{matrix}\\ 
\hline
\textrm{Calabi--Yau~functor~}\coprod_{i}k[x_i^{\pm1}]\to\mathcal C& 
\begin{matrix}\textrm{Quasi-bisymplectic~structure,} \\ \textrm{with~moment~map~}k[x^{\pm1}]\to A_{\mathcal C}\end{matrix}\\ 
\hline
\textrm{Pushing-out~along~the~``pair-of-pants''} & \textrm{Fusion}\\ 
\hline
\end{array}
\]

We want to emphasize that~\cref{secvyqb} contains what can be understood as the quasi-bisymplectic side of the fusion calculus for double 
quasi-Poisson algebra~\cite[\S5.3]{VdB}. 
Indeed, we know thanks to~\cite{VdB2} that quasi-bisymplectic structures correspond to non-degenerate double quasi-Poisson ones, 
and we produce in~\cref{fusionqH} the formula for fusion of quasi-bisymplectic structures, a noncommutative analog of~\cite[Proposition 10.7]{AKSM}. 
Because of this compatibility we do not use in this paper double quasi-Poisson structures, but prove that in the quiver case the structures we get give back Van den 
Bergh's double quasi-Poisson structures from~\cite{VdB}.

The last essential step for completeness is to check that when considering representation spaces, all these constructions yield the same symplectic structures, which 
is proved by our last main result,~\cref{complagthm}. We prove specifically that the lagrangian structures induced by quasi-Hamiltonian ones thanks 
to~\cite{VdB} on the one hand, and by relative Calabi--Yau ones~\cite{BD2} on the other hand are indeed the same. This achieves to prove the conjectural 
program established in the open questions concluding~\cite{BCS2} - except the last part which is rather independent.

\subsection*{Outline of the paper}

In \cref{section2}, we recall the mixed structure on the graded vector space of noncommutative differential forms on an associative $k$-algebra, which 
yields a convenient construction of Hochschild and negative cyclic homology as shown by Ginzburg--Schedler~\cite{GiSc}. We consider the example of 
$A=k[x^\pm]$ and identify the noncommutative differential form that yields the $1$-Calabi--Yau structure from~\cite{BCS2}. 

In \cref{section: fusion}, we compare the fusion process introduced by Van den Bergh~\cite{VdB} with certain pushouts of categories involving the pair-of-pants 
cospan studied in~\cite{BCS2}. Fusion has been introduced in order to glue idempotents in double (quasi-)Poisson algebras but in this section we only focus on the 
algebra structure and not on double brackets. Along the way, we show that the fusion of a $1$-smooth (or formally smooth, see~\cref{1sm}) algebra is $1$-smooth. 

The fourth section can be considered as an additive warm up for the next one. We show that relative Calabi--Yau structures on additive noncommutative moment 
maps induce bisymplectic structures. Bisymplectic structures where first defined in~\cite{CBEG} and are dual to non-degenerate double Poisson structures 
from~\cite{VdB}. We introduce, in analogy with Van den Bergh's fusion of double Poisson structures, the fusion of bisymplectic structures and show that 
it corresponds to composition with the additive pair-of-pants cospan from~\cite{BCS2}. Furthermore, we show that the fusion process respects the duality between 
bisymplectic and double Poisson structure in the sense that a compatible pair of bisymplectic and double Poisson structures is sent by fusion to another compatible pair 

In \cref{secvyqb} we prove that relative Calabi--Yau structures on multiplicative noncommutative moment maps induce quasi-bisymplectic structures in 
the sense of~\cite{VdB2}. Then we prove that the fusion of quasi-bisymplectic structures is induced by the composition of Calabi--Yau cospans with the multiplicative 
pair-of-pants, and that it is compatible with the duality between quasi-bisymplectic and double quasi-Poisson structures. 
We also show that in the case of multiplicative quiver varieties, the Calabi--Yau structure exhibited in~\cite{BCS2} is compatible with the non-degenerate 
double quasi-Poisson structure defined in~\cite{VdB2}.

Finally in the last section, we study the geometries induced by the aforementioned structures on representation spaces $X_V=\mathrm{Rep}(A,V)$ of 
algebras $A$ in vector spaces $V$. Namely, assuming that we have a Calabi--Yau structure on $\coprod_{i\in I}k[x^{\pm1}] \to \cC$, with $A_\cC=A$, 
we know thanks to~\cite{BD2} that it induces a lagrangian structure on $[X_V/\mathrm{GL}_V]\to[\mathrm{GL}_V/\mathrm{GL}_V]$. We also know 
that the double quasi-Poisson structure induced by 
our previous section yields a quasi-Hamiltonian structure on $X_V$ (in the sense of~\cite{AMM}), and therefore a lagrangian structure on the very same morphism. 
We prove that these two lagrangian structures match.

\subsubsection*{Related works}

A systematic comparison of noncommutative differential forms with Hochschild and cyclic complexes have been achieved by Yeung in~\cite{Yeung}.
There the author uses~\cite{GiSc0}, while we rely on~\cite{GiSc}.
We should also mention Pridham's~\cite{Pridham}, where a systematic way of producing shifted bisymplectic (resp.~bilagrangian) structures out of
absolute (resp.~relative) Calabi--Yau structures (see Proposition 1.24 and Theorem 1.56 in~\cite{Pridham}). One may be able to recover some of the
results of the present paper using Pridham's general theory (but it would probably require as much work as here to derive these results from~\cite{Pridham}).

\subsection*{Acknowledgements}

We thank Maxime Fairon for discussions about double brackets. We also learned a lot about those during the Villaroger 2021 
workshop on double Poisson structures, of which we thank all the participants. 
The first and second author have received funding from the European Research Council (ERC) under the
European Union's Horizon 2020 research and innovation programme (Grant Agreement No.\ 768679).


\section{Cyclic and noncommutative de Rham mixed complex}\label{section2}

In this section we first briefly recall some facts about Hochschild and negative cyclic homology, and then some constructions and results from~\cite{GiSc}.  
In particular, in~\cite{GiSc} Ginzburg and Schedler directly relate the negative cyclic homology of a unital algebra with the cohomology of a 
complex that is obtained from the mixed complex of noncommutative differential forms~\cite{Kar} on this algebra. We finally exhibit a closed 
noncommutative form representing the class in negative cyclic homology which defines the $1$-Calabi-Yau structure on $k[x^{\pm1}] $ in \cite{BCS2}.  

\subsection{Hochschild and negative cyclic homology}

We denote by $\cat{Mod}_k$ the category of \emph{chain} complexes over $k$. 
We warn the reader that we use the homological grading instead of the cohomological grading used in our previous papers~\cite{BCS,BCS2}. 
In particular, differentials have degree $-1$ while mixed differential have degree $+1$. 
Apart from this change, all along this paper we borrow the convention and notation from \textit{op. cit.}, to which we refer for more details. 
For instance, whenever $\cat{M}$ is a model category we write $\scat{M}$ 
for the corresponding $\infty$-category obtained by localizing along weak equivalences. 

A \textit{dg-category} is a $\cat{Mod}_k$-enriched category and the category of dg-categories with dg-functors is denoted by $\cat{Cat}_k$. 
We refer to \cite{KellerDG,ToDG2} for a detailed introduction to dg-categories and their homotopy theory.  
The \textit{Hochschild chains} $\infty$-functor is then defined as 
\[
\cat{HH}\,:\,\scat{Cat}_k\longrightarrow \scat{Mod}_k\,;\,
\cC\longmapsto \cC\underset{\cC^e}{\overset{\mathbb{L}}{\otimes}}\cC^{\op}\,,
\]
where $\cC^e:=\cC\otimes\cC^{\op}$. We write $\cat{HH}_{i}(\cC)$ for the $i$-th homology of $\cat{HH}(\cC)$. 

There is an explicit description of the derived tensor product $\cC\underset{\cC^e}{\overset{\mathbb{L}}{\otimes}}\cC^{\op}$, which uses the 
normalized bar resolution of $\cC$ as a $\cC$-bimodule, and that leads to standard normalized Hochschild chains, that we denote $\big(C_{*}(\cC),b\big)$: 
\[
C_{*}(\cC)=\bigoplus_{\substack{n\geq0 \\ a_0,\dots, a_n\in\mathrm{Ob}(\cC)}}
\cC(a_n,a_0)\otimes\bar\cC(a_{n-1},a_n)\otimes\cdots\otimes \bar\cC(a_1,a_2)\otimes \bar\cC(a_0,a_1)[-n], 
\]
with $\bar\cC(a,a')=\cC(a,a')$ if $a\neq a'$ and $\bar\cC(a,a)=\cC(a,a)/k\cdot\mathrm{id}_a$. 

\medskip

Hochschild chains carry a mixed structure, that is given on the standard normalized model by Connes's $B$-operator. 
We refer to \cite{BCS,BCS2} and references therein, for the homotopy theory of mixed complexes, and explicit formulas\footnote{Beware of the change of 
(co)homological grading convention, though. }. 
The negative cyclic complex of $\cat{C}$, denoted by $\cat{HC}^-(\cat{C})$, is defined as the homotopy fixed points of $\cat{HH}(\cC)$ with respect 
to the mixed structure; it comes with a natural transformation $(-)^\natural: \cat{HC}^- \Rightarrow \cat{HH}$. 
In concrete terms, $\cat{HC}^-(\cat{C})$ is given by $\big(C_{*}(\cC)[\![u]\!],b-uB\big)$, where $u$ is a degree $-2$ variable. 

\medskip

We can view every dg-algebra with a finite set $(e_i)_{i\in I}$ of orthogonal nonzero idempotents such that $1=\sum_{i\in I}e_i$ as a dg-category 
with object set $I$. Conversely, we can associate to every dg-category $\cC$ with finitely many objects its \textit{path algebra} given by the complex 
\[
A_{\cC}:=\bigoplus_{(a, b) \in Ob(\cC)\times Ob(\cC)} \cC(a, b)
\]
with product given by composition of morphisms. The dg-algebra $A_{\cC}$ is an $R$-algebra, where $R=\oplus_{c\in Obj(c)}ke_c$
Note that the construction is in general not functorial, meaning that a functor does not necessary give a morphism between the corresponding dg-algebras 
(unless the functor is injective on objects). This can be seen very easily on the following example, which will play an important role in the next section. 
\begin{example}
The dg-category coproduct $k\coprod k$ is the dg-category given by two objects $1$ and $2$ and endomorphism ring $k=\mathrm{End}(1)$ 
respectively $k=\mathrm{End}(2)$ at each object but zero Hom-spaces between the two objects. Hence its path algebra $A_{ k \coprod k}$ is isomorphic to 
$k \oplus k$. There is a dg-functor 
\[
k \coprod k \to k 
\]
sending $1$ and $2$ to $pt$, which denotes the only object of $k$, but there is no map of $k$-linear dg-algebras $ k \oplus k \to  k$. 
\end{example} 
Nevertheless, $\cC$ and $A_\cC$ are Morita equivalent, so that their Hochschild (resp.~negative cyclic) homology are isomorphic. 
More precisely, we have an inclusion of mixed complexes $\big(C_{*}(\cC),b,B\big)\hookrightarrow \big(C_{*}(A_\cC),b,B\big)$, which is a weak equivalence 
(here we view $A_\cC$ as a dg-category with one object). 

\subsection{Noncommutative forms}

Consider a unital associative $k$-algebra $A$, along with a subalgebra $R$.
We fix a complementary subspace $\bar A\simeq A/R$ of $R$. Denote by $d:A\to\bar A$ the associated quotient map. 
We will systematically use the $\bar{~}$ notation for the quotient by $R$.
The graded algebra $\Omega_R^* A$ of noncommutative differential forms is defined as the quotient of $T_R(A\oplus \bar A[-1])$ by the relations
\[
a\otimes b=ab\qquad\text{and}\qquad d(ab)=a\otimes d(b)+d(a)\otimes b
\]
for every $a,b\in A$. It comes equipped with a mixed differential, that is the derivation induced by $d$, and that we denote by the same symbol. 
The mixed differential $d$, descends to the \emph{Karoubi--de Rham} graded vector space $\mathrm{DR}^*_R A:=\Omega^*_R A/[\Omega^*_R A,\Omega^*_R A]$, 
first introduced in~\cite{Kar}. 

\medskip

In order to define a differential on $\Omega^*_R A$, turning it into a mixed complex, we consider the distinguished double derivation 
$E:a\mapsto a\otimes 1-1\otimes a$, denoted by $\Delta$ in~\cite{CBEG}. 
Recall that the $A$-bimodule of ($R$-linear) double derivations is defined as  
\[
D_{A/R}:=\mathrm{Der}_R(A,A\otimes A)\simeq \Omega_R^1 A^\vee, 
\]
where the derivations are taken with respect to the outer $A$-bimodule structure on $A\otimes A$, and the remaining 
$A$-bimodule structure on $D_{A/R}$ comes from the inner one on $A\otimes A$. 
Here $\Omega_R^1A$ is the kernel of the multiplication $A\otimes_R A\to A$, and inherits its $A$-bimodule structure from 
the outer one on $A\otimes A$; it is isomorphic to $A\otimes_R\bar A$ as a left $A$-module ($1\otimes da\in A\otimes_R\bar A$ 
being identified with $E(a)\in \Omega_R^1A$. As a matter of notation, we will often write $\Omega_{A/R}:=\Omega_R^1 A$. 

There is an obvious graded algebra isomorphism $\Omega^*_R A\simeq T_A(\Omega_R^1A[-1])$, as well as a left $A$-module isomorphisms 
$\Omega_R^n A\simeq A\otimes_R\bar A^{\otimes_R n}$ (see~\cite{CQ}). 
For later purposes, we also introduce the graded algebra of polyvector fields $D^*_RA=T_A(D_{A/R}[-1])$ from~\cite{VdB}. 

Following~\cite{CBEG} we define, for any $R$-linear double derivation $\delta\in D_{A/R}$ of $A$, a graded double derivation
\[
i_\delta:\Omega^*_R A\rightarrow \Omega^*_R A \otimes \Omega^*_R A
\]
of $\Omega^*_R A$ by setting
\[
i_\delta(a):=0\qquad \mathrm{and}\qquad i_\delta(da):=\delta(a)
\]
for any $a\in A$. On $\Omega_R^2A$ we thus have for instance
\[
i_{\delta}(pdqdr)=p\delta(q)'\otimes\delta(q)''dr-pdq\delta(r)'\otimes\delta(r)''\in A\otimes\Omega_R^1A+\Omega_R^1A\otimes A,
\]
where we use Sweedler's sumless notation $\delta(a)=\delta(a)'\otimes\delta(a)''$. 
The graded double derivation $i_\delta$ induces a linear contraction operator 
\[
\iota_\delta:={}^\circ i_\delta:\Omega_R^* A\rightarrow\Omega_R ^{*-1}A,
\]
where ${}^\circ(\alpha\otimes\beta)=(-1)^{kl}\beta\otimes\alpha$ for $\alpha\otimes\beta\in\Omega_R^kA\otimes\Omega_R^lA$. 
Our differential will be given by the contraction operator $\iota_E:\Omega_R^* A\to\Omega_R ^{*-1}A$, which has the following properties thanks 
to ~\cite[Lemma 3.1.1]{CBEG}: it is explicitely given by the formula
\[
\iota_E(a_0da_1\dots da_n)=\sum_{l=1}^n(-1)^{(l-1)(n-1)+1}[a_l,da_{l+1}\dots da_na_0da_1\dots da_{l-1}],
\]
it vanishes on $[\Omega^*_RA,\Omega^*_RA]$ (and thus factors though $\mathrm{DR}_R^* A$) it takes vales in 
$[\Omega^*_RA,\Omega^*_RA]^R$ (in particular, $\iota_E^2=0$), and $[\iota_E,d]=0$. As a consequence, we obtain that 
$\big(\Omega^*_RA,\iota_E,d)$ is a mixed complex. 

\subsection{Hochschild chains versus noncommutative forms}\label{subsechhncform}

Below we rephrase some constructions and results of~\cite{GiSc} in terms of mixed complexes. 
Beware that the notation used here is not exactly the same than in \textit{op. cit.}. 
For the moment we only assume that $A$ is a $k$-algebra. 

Through the identification $C_*(A)\simeq \Omega^*_kA$, the Hochschild differential $b$ reads as
\[
b(\alpha da)=(-1)^{|\alpha|}[\alpha,a]. 
\]
The Karoubi operator on $\Omega_k^* A$, 
given by
\[
\kappa(\alpha da)=(-1)^{|\alpha|}da\alpha,
\]
allows one to define a harmonic decomposition $\bar\Omega_k^* A=P\bar\Omega_k^* A\oplus P^\perp\bar\Omega_k^*A$ where
\[
P\bar\Omega_k^* A=\ker(1-\kappa)^2\qquad\text{and}\qquad P^\perp\bar\Omega_k^*A=\mathrm{ima}(1-\kappa)^2.\]
The following identites hold: 
\[ 
\iota_E=bN|_P\qquad\text{and}\qquad B=Nd|_P,
\] 
where $N$ is the grading operator and $B$ the Connes mixed differential. 

Hence we have the following chain of morphisms of mixed complexes
\begin{equation}\label{NCHKR}
\xymatrix{
(\bar\Omega_k^*A, \iota_E, d)\ar[r]^-P&(P\bar\Omega_k^*A,\iota_E,d)\ar[r]^-{N!}&(P\bar\Omega_k^*A,b,B)\ar@{^(->}[r]&(\bar\Omega_k^*A,b,B)
}
\end{equation}
such that, according to \cite{GiSc}: $\overline{[d\Omega_k^*A,d\Omega_k^*A]}\hookrightarrow(\ker(P)[\![u]\!],\iota_E-ud)$ is a quasi-isomorphism, 
where $u$ is a degree $-2$ formal variable, $N!$ is an isomorphism, and the rightmost inclusion is a quasi-isomorphism. 
We thus get a quasi-isomorphism
\[
\left(\dfrac{\bar\Omega_k^* A[\![u]\!]}{\overline{[d\Omega_k^*A,d\Omega_k^*A]}},\iota_E-ud\right) \longrightarrow (\bar\Omega_RA[\![u]\!],b-uB)
\]
and the homology of both complexes yields the reduced negative cyclic homology $\overline{\mathrm{HC}}^-(A)$. 

\medskip

Hence, when $A=A_\cC$, for $\cC$ a genuine $k$-linear category with a finite set $I$ of objects, and 
$R=\oplus_{i\in I}ke_i$, we have a zig-zag \[
\xymatrix{
~~\bigg(\bar C_*(\cC)[\![u]\!],b-uB\bigg)~~\ar@{^(->}[rr]^-{\sim}&& ~~\bigg(\bar C_*(A)[\![u]\!],b-uB\bigg)~~\\
\left(\dfrac{\bar\Omega_k^* A[\![u]\!]}{\overline{[d\Omega_k^*A,d\Omega_k^*A]}},\iota_E-ud\right)\ar[urr]^-[@!15]\sim\ar[rr]&&
 \left(\dfrac{\bar\Omega_R^* A[\![u]\!]}{\overline{[d\Omega_R^*A,d\Omega_R^*A]}},\iota_E-ud\right)
}\]
where only the last bottom arrow may not be a quasi-isomorphism. 
\subsection{Computations for $A=k[x^{\pm1}]$}\label{1cypm}

As a matter of convention, we always mean $(dx)y$ if no brackets appear in $dxy$.
We want to find a harmonic cyclic lift for $\alpha_1:=x^{-1}dx\in\bar\Omega^1A$ which is closed for the mixed structure $(P\bar\Omega,\iota_E,d)$. 
That means that $A$ is $1$-pre-Calabi--Yau according to the terminology of~\cite{BCS}. This was already proved in~\cite{BCS2} using the standard normalized 
Hochschild complex, but we reprove it here on the ``de Rham side'' and check consistency afterwards to illustrate~\eqref{NCHKR}.

Set $\alpha_n=(x^{-1}dx)^{2n-1},\beta_n=\kappa(\alpha_n)=(dxx^{-1})^{2n-1}\in\bar\Omega^{2n-1}A$. Then\[
\kappa(\beta_n)=\kappa(-\beta_{n-1}dxdx^{-1})=-dx^{-1}\beta_{n-1}dx=\alpha_n.\]
Hence $\alpha_n+\beta_n\in P\bar\Omega A$ and $\alpha_n-\beta_n=\frac{1}{2}(1-\kappa)^2(\alpha_{n})\in P^\perp\bar\Omega A$.
Then,\begin{align*}
\iota_E\alpha_n&=\frac{1}{2}(2n-1)b(\alpha_n+\beta_n)\\
&=\frac{1}{2}(2n-1)([\alpha_{n-1}x^{-1}dxx^{-1},x]-[\beta_{n-1}dx,x^{-1}])\\
&=\frac{1}{2}(2n-1)(x^{-1}\beta_{n-1}dx+\alpha_{n-1}x^{-1}dx-\beta_{n-1}dxx^{-1}-x\alpha_{n-1}x^{-1}dxx^{-1})\\
&=(2n-1)((x^{-1}dx)^{2n-2}-(dxx^{-1})^{2n-2}).
\end{align*}
On the other hand, $d\alpha_1=-(x^{-1}dx)^2$, and if we assume $d\alpha_{n-1}=-(x^{-1}dx)^{2n-2}$, we get\begin{align*}
d\alpha_n&=d(x^{-1}dx(x^{-1}dx)^{2n-2})\\
&=d(x^{-1}dx)(x^{-1}dx)^{2n-2}-x^{-1}dxd((x^{-1}dx)^{2n-2})\\
&=-x^{-1}dxx^{-1}dx(x^{-1}dx)^{2n-2}-x^{-1}dxd^2\alpha_{n-1}\\
&=-(x^{-1}dx)^{2n}.
\end{align*}
Similarly $d\beta_n=(dxx^{-1})^{2n}$ for all $n$. Thus, as $\iota_E\alpha_n=\iota_E\beta_n$,\[
\iota_E(\alpha_n+\beta_n)=2\iota_E\alpha_n=-2(2n-1)d(\beta_{n-1}+\alpha_{n-1}).\]
As a consequence $(\iota_E-ud)(\gamma)=0$, where $\gamma_k=\frac{1}{2}(\alpha_k+\beta_k)\in P\bar\Omega^{2k-1}k[x^{\pm1}]$ and
\[\gamma=\sum_{k\ge0}\dfrac{k!}{(2k+1)!}(-u)^k\gamma_{k+1}\]
where $u$ is a formal degree $-2$ variable.

Let us check now that it is coherent with~\cite{BCS2}.
Through~\eqref{NCHKR} and the isomorphism $\Omega^nA\simeq A\otimes\bar A^{\otimes n}$, $\gamma$ is mapped to \[
\sum_{k\ge0}{k!}u^k\dfrac{(x^{-1}\otimes x)^{\otimes(k+1)}-(x\otimes x^{-1})^{\otimes(k+1)}}{2}\]
as \begin{align*}
\alpha_{k+1}&=(x^{-1}dx)^{2k+1}=(-1)^kx^{-1}(dxdx^{-1})^kdx,\\
\beta_{k+1}&=(dxx^{-1})^{2k+1}=(-1)^{k+1}x(dx^{-1}dx)^kdx^{-1},\\
\text{and }\gamma_{k+1}&\in P\bar\Omega^{2k+1},
\end{align*}
all of which being consistent with~\cite[3.1.1]{BCS2}.


\section{Fusion}\label{section: fusion}

In this section, we compare certain pushouts of $k$-linear dg-categories with the fusion formalism introduced by Van den Bergh~\cite{VdB} for algebras. 
Fusion is a process which glues two pairwise orthogonal idempotents into one. Given an algebra with a double (quasi-)Poisson structure, the 
new algebra obtained by fusion inherits a double (quasi-)Poisson structure from the original one as shown in~\cite{VdB,Fairon}. 

This will be relevant in the next sections, where we will compare fusion of bisymplectic and quasi-bisymplectic structures 
with compositions of Calabi--Yau cospans. 
  
\subsection{Fusion as a pushout}\label{fusion3.1}
 
Recall that Van den Bergh defines in~\cite{VdB} the fusion algebra which identifies two pairwise orthogonal idempotents. 
We use the notation $(-)^+$ instead of $\overline{(-)}$ as in~\cite{VdB} since it is already used.
 
\begin{definition}\label{def: fusion}
Let $R= ke_1 \oplus \cdots \oplus k e_n$ be a semi-simple algebra with pairwise orthogonal idempotents $e_i$, and $A$ an 
$R$-algebra. Set $\mu=1-e_1-e_2$ and $\epsilon=1-e_2$. Then the fusion algebra $A^f$ is defined as $\epsilon {A}^+ \epsilon$, where 
${A}^+:= A \coprod_{ke_1\oplus ke_2\oplus k\mu} (M_2(k) \oplus k\mu)$. 
Here $M_2(k)$ denotes the $(ke_1\oplus ke_2)$-algebra of $2\times 2$ matrices, and the idempotent $e_i$ is sent to $e_{ii}$, where $e_{i j}$'s are matrix units. 
\end{definition}
One can see that ${A}^+$ is isomorphic to $A \coprod_{R} {R^+}$, and that $R^+=M_2(k)\oplus R_{\geq3}$ and $R^f=ke_1\oplus R_{\geq3}$, where 
$R_{\geq3}:=ke_3\oplus\cdots\oplus ke_n$. 

\medskip

Let now $\cC$ be a dg-category with a finite set of objects $I=\{1,\dots,n\}$, $n\ge2$. We define 
\[
\cC^f:=\cC \coprod_{k\coprod k} k,
\]
where the functor $k\coprod k\to \cC$ is given by the units of the first two objects $1$ and $2$. 
Note that the strict pushout is (categorically equivalent to) a homotopy pushout. 

\begin{examples}
(1) The category $(k[x]\amalg k[y])^f$ (when defined using the strict pushout) is isomorphic to $k\langle x,y\rangle$. 
Similarly, $(k[x^{\pm1}]\amalg k[y^{\pm1}])^f$ is isomorphic $k\langle x^{\pm},y^{\pm1} \rangle$. 
As a consequence, we get that
\[
\cC^f\simeq \cC \coprod_{k[x_1^\square] \coprod k[x_2^\square]} k\langle x_1^\square,x_2^\square \rangle,
\]
where $\square\in\{\emptyset,\pm1\}$ and $k[x_i^\square] \to\mathrm{End}_{\cC}(i)$.  

(2) If $\mathcal R= \coprod_{i\in I} k$, then $\mathcal R^f= k \amalg \mathcal R_{\geq3}$, where $\mathcal R_{\geq3}:=\coprod_{i\geq3} k$. 
As a consequence, we get that 
\[
\cC^f:=\cC \coprod_{\mathcal R} \big( k\amalg \mathcal R_{\ge3}\big),
\]
where the functor $\mathcal R  \rightarrow \cC$ is uniquely determined by mapping the object of the $i$-th copy of $k$ to $i$, and the functor 
$\mathcal R \to  k\amalg \mathcal R_{\ge3}$ maps sends the first two objects of $\mathcal R$ to the object of the first copy of $k$. 
\end{examples}

\begin{proposition}\label{prop: fusion}
Let $\cC$ be a $k$-linear dg-category with set of objects $I$. Then  $A_{\cC^f}$ is isomorphic to $ (A_{\cC})^f$. 
\end{proposition}

\begin{proof} 
We can assume without loss of generality that $\cC$ has only two objects $1$ and $2$. We denote $e_1$ and $e_2$ their respective identity map.  
The dg-category $\cC \coprod_{k\coprod k} k$ has exactly one object which we denote $pt$. Let us show that the endomorphism ring 
$B:=\mathrm{End}(pt)$ is isomorphic to the fusion algebra $A^f$ of $A:=A_{\cC}$. 
By the pushout property, there are algebra homomorphisms 
\begin{align*}
f:\mathrm{End}_\cC(1) &\simeq e_1 A e_1  \to B\\
g:\mathrm{End}_\cC(2) &\simeq e_2 A e_2  \to B, 
\end{align*}
and bimodule morphisms $ e_1A e_2 \simeq \cC(2, 1) \to B, e_1 a e_2 \mapsto e_1 a e_{21}$ and 
$e_2A e_1 \simeq \cC(1, 2) \to B, e_2 a e_1 \mapsto e_{12} a e_1$ such that 
\[
\xymatrix{ 
\cC(2, 1) \otimes \cC(1, 2) \ar[d] ^{m} \ar[r] & B \otimes B \ar[d]^m\\
\mathrm{End}_\cC(1) \ar[r]^g & B }
\]
commutes. The algebra homomorphism $k \to B$ is then uniquely determined. 

We have injective algebra morphisms $ \mathrm{End}_\cC(1) \simeq e_1 A e_1  \to A^f, a \mapsto a$, 
$\mathrm{End}_\cC(2) \simeq e_2 A e_2  \to A^f, a \mapsto e_{12} a e_{21}$. Similarly, we have injective morphisms of bimodules 
$\cC(2,1) \simeq e_1 A e_2 \to A^f, a \mapsto a e_{21} $ and $\cC(1,2) \simeq e_2 A e_1 \to A^f, a \mapsto e_{12} a$ compatible with the composition 
of morphisms. Hence we obtain a unique injective algebra homomorphism $B \to A^f$. As the image of the above maps generate $A^f$, this morphism is 
also surjective and hence $B=A_{\cC^f} \simeq A^f$. 
\end{proof}

\subsection{Trace maps}\label{fusion3.2}

Acccording to Van den Bergh~\cite{VdB} we consider the following situation: an $R$-algebra $A$, and an idempotent $e$ in $R$ such that $ReR=R$. 
One writes $1= \sum_i p_i e  q_i$ with $p_i, q_i\in R$, and define a \textit{trace map}\[
\mathrm{Tr}:A\rightarrow eAe~;~a\mapsto\sum_ieq_iap_ie.\]
We recall a series of standard results, for which we provide full proofs for the sake of completeness; the main point is to be able to describe the trace map 
on $\Omega_RA$ and $\mathrm{DR}_RA$. 

\begin{lemma}\label{lemmedetrace}
The trace map $\mathrm{Tr}$ descends to an isomorphism $A/[A,A]\to eAe/ [eAe,eAe]$ that does not depend on the choice of decomposition $1= \sum_i p_i e  q_i$. 
\end{lemma}
\begin{proof}
First of all, the trace map $\mathrm{Tr}$ sends commutators to commutators. Indeed: 
\begin{align*}
\mathrm{Tr}(ab-ba) & = \sum_{i}(eq_iabp_ie-eq_ibap_ie)  \\
& = \sum_{i,j}eq_iap_jeq_jbp_ie-eq_ibp_jeq_jap_ie \\
& = \sum_{i,j}eq_iap_jeq_jbp_ie-eq_jbp_ieq_iap_je\in[eAe,eAe]
\end{align*}
Then, one can check that it is a $k$-linear inverse modulo commutators, to the algebra morphism $eAe\to A$. Indeed: on the one hand, 
$a=\sum_ip_ieq_i a=\mathrm{Tr}(a)~\mathrm{mod}~[A,A]$, and on the other hand, 
$eae=\sum_iep_ieq_i eae=\mathrm{Tr}(eae)~\mathrm{mod}~[eAe,eAe]$. 
Since the morphism $eAe\to A$ does not depend on the decomposition of $1$, its inverse (modulo commutators) does not either. 
\end{proof}

\begin{lemma}\label{lemma: trace}
For any two $A$-bimodule $M$ and $N$, the canonical morphism $Me \otimes_{eRe} eN \to M \otimes_R N$ of $A$-bimodules is inversible with inverse given by 
\[
\Psi_{M,N}:M \otimes_R N \to Me \otimes_{eRe} eN~;~m \otimes n \mapsto \sum_i mp_i e \otimes eq_i n.
\]
\end{lemma}

\begin{proof}
Let us check that it is well defined. Consider $r\in R$ and write $r= \sum_j h_j e l_j$ for some $h_j, l_j \in R$. Then \allowdisplaybreaks
\begin{align*} 
\Psi_{M,N} (mr \otimes n ) & = \sum_imrp_i e \otimes eq_i n =  \sum_{i,j} m  h_j e l_j p_i e \otimes eq_i n\\
&=  \sum_{i,j} m  h_je   \otimes e l_j p_i eq_i n =  \sum_{j} m  h_je   \otimes e l_j  n\\
&= \sum_{i,j} m  p_i e q_i h_je   \otimes e l_j  n =  \sum_{i,j} m  p_i  e \otimes e q_i  h_je   l_j  n \\
&=\sum_i  m  p_i  e \otimes e q_i r n  = \Psi_{M,N}( m \otimes r n ).
\end{align*}
We finally observe that $\Psi_{M,N}$ is an inverse to the canonical morphism $Me \otimes_{eRe} eN \to M \otimes_R N$. Indeed, in $M \otimes_R N$, 
$\sum_i mp_i e \otimes eq_i n=\sum_i m\otimes p_i e q_i n=m\otimes n$, and in $Me \otimes_{eRe} eN$, 
$\sum_i mep_i e \otimes eq_i en=\sum_i me\otimes p_i e q_i e n=me\otimes en$. 
\end{proof} 

As a matter of notation, we introduce $\Psi_M:=\Psi_{M,M}$. 

\begin{lemma}\label{lemma: trace-omega} 
The isomorphism $\Psi_{\Omega_{A/ R}}$ induces an isomorphism $e(\Omega_RA)e \simeq \Omega_{eRe}(eAe)$, through which the trace map 
of $\Omega_RA$ reads as follow: 
\begin{align*}
 \mathrm{Tr}:  \Omega_RA& \rightarrow e(\Omega_RA)e \simeq \Omega_{eRe}(eAe) \\
  a_0da_1\dots da_m&\longmapsto \sum_{i_0,\dots,i_m}eq_{i_0}a_0p_{i_1}ed(eq_{i_1}a_1p_{i_2}e)\dots d(eq_{i_m}a_mp_{i_0}e). 
\end{align*}
Moreover, it induces a $k$-linear isomorphism 
\[
\mathrm{Tr}: \DR_R ( A ) \to \DR_{eRe}(  eAe ) 
\]
that does not depend on the decomposition $1=\sum_i p_ieq_i$. 
\end{lemma}

\begin{proof}
Thanks to the previous lemma, the isomorphism $\Psi_{\Omega_{A/ R}}$ induces an isomorphism of tensor algebras 
$e( T_{A} \Omega_{A/R} )e\simeq T_{eAe}  ( e\Omega_{A/ R} e) $. 
Using $\Psi_A$ we also have 
\begin{align*}
\Omega_{eAe/eRe}&=\ker(eAe\otimes_{eRe}eAe\rightarrow eAe)\\
&\simeq\ker(eA\otimes_{R}Ae\rightarrow eAe)\\
&=e\ker(A\otimes_{R}A\rightarrow A)e\\
&=e\Omega_{A/R}e.
\end{align*}
Combining these we get 
\[
e (\Omega_RA) e := e( T_{A} \Omega_{A/R} )e\simeq T_{eAe}  ( e\Omega_{A/ R} e)\simeq T_{eAe}   \Omega_{eAe/ eRe} = :\Omega_{eRe}(eAe).
\]
Through this identification, an element $edae=ea\otimes e-e\otimes ae\in e\Omega_{A/R}e$ becomes, in $\Omega_{eAe/eRe}$, 
\[
\sum_ieap_ie\otimes eq_ie-ep_ie\otimes eq_iae=eae\otimes e-e\otimes eae=:d(eae)\in \Omega_{eAe/eRe}. 
\]
Thus the trace map reads
\begin{align*}
\Omega_RA\ni a_0da_1\dots da_m&\mapsto\sum_{i_0}eq_{i_0}a_0da_1\dots da_mp_{i_0}e\in e(\Omega_RA)e\\
&\mapsto\sum_{i_0,i_1,\dots,i_m}eq_{i_0}a_0p_{i_1}ed(eq_{i_1}a_1p_{i_2}e)\dots d(eq_{i_m}a_mp_{i_0}e)\in \Omega_{eRe}(eAe).
\end{align*}
The last part of the claim follows from \cref{lemmedetrace}
\end{proof} 

\subsection{Functoriality}

We now apply the constructions from the previous \cref{fusion3.2} to the idempotent $\epsilon=1-e_2$ of ${R^+}$ (see \cref{def: fusion}), 
where $1=\epsilon\epsilon\epsilon+e_{21}\epsilon e_{12}$. 
Precomposing with the algebra morphism $A\to{A^+}$ we get maps $\Omega_RA\to\Omega_{R^f}A^f$ and $ \DR_R(A)  \to \DR_{R^f}A^f$ that we denote by $(-)^f$. 
Since $\epsilon e_{12}=e_{12}$ and $e_{21}\epsilon=e_{21}$, we have $\mathrm{Tr}(a)= \epsilon a \epsilon+ e_{12} a e_{21}$ for all $a\in{A^+}$. 
Actually the trace map in this situation also has a simpler expression on forms.

\begin{lemma}\label{lemma: tromega}
On $\Omega_{{A^+}/{R^+}}$ we have\[
\mathrm{Tr}(adb )=  \epsilon  a  d b \epsilon +  e_{12} ad  b e_{21}\]
and dually we have a trace map on double derivations \[
\mathrm{Tr}:D_{{R^+}}{A^+}\to D_{R^f}A^f~,~\delta\mapsto \epsilon  \delta \epsilon +  e_{12} \delta e_{21}.\]
More generally, if $\omega\in\Omega_{{R^+}}{A^+}$, we have $\mathrm{Tr}(\omega)=\epsilon  \omega \epsilon +  e_{12} \omega e_{21}$.
\end{lemma}

\begin{proof}
Thanks to~\cref{lemma: trace-omega} we have on $1$-forms\begin{align*}
\mathrm{Tr}(adb )&= \epsilon a \epsilon d (\epsilon b \epsilon)+e_{12} a \epsilon d(\epsilon b e_{21}) + e_{12} a e_{21} d (e_{12} b e_{21})+  \epsilon  a e_{21}d ( e_{12} b \epsilon) \\
&= \epsilon a \epsilon d  b \epsilon+e_{12} a \epsilon d b e_{21}+ e_{12} a e_{2} d  b e_{21}+  \epsilon  a e_{2}d b \epsilon.
\end{align*}
If $a\in Ae_2$ and $b\in e_2A$, as $\epsilon e_2=e_2\epsilon=0$, we get\[
\mathrm{Tr}(adb )= e_{12} ad  b e_{21}+  \epsilon  a  d b \epsilon.\]
If $a\in Ae_i$ and $b\in e_iA$ for some $i\neq 2$, as $\epsilon e_i=e_i\epsilon=e_i$ we again have\[
\mathrm{Tr}(adb )=  \epsilon  a  d b \epsilon +  e_{12} ad  b e_{21}.\]
It generalizes to all forms.
\end{proof}

We go back to the context of a dg-category $\cC$ with finite set of objects $I$, an set $A:=A_\cC$. 
We define idempotents $e_i=\mathrm{id_i}$ and set $R=\oplus_{i\in I}ke_i$, a subalgebra of $A$. Recall that $R^f\simeq \oplus_{i\neq2}ke_i$, 
and consider the $k$-linear map $C_{*}(\cC)\to \Omega^*_R A$ given by
\[
a_0\otimes a_1\otimes\dots\otimes a_m\mapsto a_0da_1\dots da_m.
\]
Since there is a functor $\cC\to \cC^f$, we have a natural map $\nu:C_{*}(\cC)\to C_{*}(\cC^f)$.

\begin{lemma}\label{lemma: HH and fusion}
The following diagram commutes
\[
\xymatrix{ C_{*}(\cC) \ar[r] \ar[d]_-\nu & \Omega_R^*(A)\ar[d]^{(-)^f} \\
C_{*}(\cC^f)\ar[r] &\Omega_{R^f}^*(A^f).}
\] 
\end{lemma} 

\begin{proof}
Thanks to~\cref{lemma: trace-omega} the map $ \Omega_R^*(A) \to  \Omega_{R^f}^*(A^f)$ is given by
\[
(a_0d a_1\cdots a_m)^f=\sum_{i_0,\dots,i_m}q_{i_0}a_0p_{i_1}d(q_{i_1}a_1p_{i_2})\dots d(q_{i_m}a_mp_{i_0}). 
\]
Since $p_{i_j}\epsilon=p_{i_j}$ and $\epsilon q_{i_j}=q_{i_j}$ in our situation, that is either $p_{i_j}=\epsilon=q_{i_j}$ or $p_{i_j}=e_{21},q_{i_j}=e_{12}$. 
Now, if $a_0\otimes\cdots a_m$ belongs to the Hochschild complex of $\cC$, then these elements are completely 
determined by the $a_j$'s: indeed, if $a_j\in\cC(x_{j+1},x_j)$ then $q_{i_j}=\epsilon$ whenever $x_j\neq 2$ and $p_{i_{j+1}}=\epsilon$ whenever $x_{j+1}\neq 2$. 
  
From the proof of~\cref{prop: fusion} we have that $\cC(x,y)\to A^f$ is given by $a\mapsto q a p$, with 
\begin{itemize}
\item $q=\epsilon$ if $y\neq 2$, and $e_{12}$ otherwise. 
\item $p=\epsilon$ if $x\neq 2$, and $e_{21}$ otherwise. 
\end{itemize}
Hence the composed map $ C_{*}(\cC)\to  C_{*}(\cC^f)\to \Omega_{R^f}^*(A^f)$ is given by 
\[
a_0 \otimes \cdots \otimes a_m \mapsto q_{i_0}a_0p_{i_1} \otimes q_{i_1} a_2 p_{i_2} \otimes \cdots \otimes q_{i_m} a_m p_{i_0},
\]
with the same $p_{i_j}$'s and $q_{i_j}$'s as above, proving the commutativity. 
\end{proof}

\begin{lemma}\label{lemma: trace functoriality} 
Le $\omega \in \Omega^2_R(A)$. Then $\omega$ induces a map $\iota(\omega): D_{A/R}\to \Omega_{A/R}$. Under the fusion process, the following diagram commutes
 \[
 \xymatrix{ D_{A/R} \ar[r] \ar[d]^{\iota(\omega)} & D_{{A^+}/{R^+}} \ar[r] \ar[d]^{\iota({\omega^+}) } &  D_{A^f/R^f} \ar[d]^{ \iota(\mathrm{Tr}(\omega^+))=\iota(\omega^f) } \\
 \Omega_{A/R}  \ar[r]&  \Omega_{{A^+}/ {R^+}}\ar[r]&  \Omega_{A^f/ R^f}.}
 \]
\end{lemma} 
\begin{proof}
The commutativity of the left hand side square follows immediately from definitions and the commutativity of the right hand side square means that 
\[
\iota_{\mathrm{Tr}(\delta)} (\mathrm{Tr}(\omega) ) = \mathrm{Tr} (\iota_{\delta} (\omega) )
\]
for all $\omega \in  \Omega^2_{{R^+}}({A^+})$ and $\delta \in D_{{A^+}/{R^+}}$. We prove this now.
Recall that the bimodule structure on $D_{A/R}$ is induced by the inner one on $A\otimes_RA$. 
We know from the proof of~\cite[Lemma 2.8.6]{CBEG} 
that $\iota_{a\delta b}=a\iota_\delta b$. We thus have, thanks to~\cref{lemma: tromega},
\begin{align*}
\iota_{\mathrm{Tr}(\delta)} (\mathrm{Tr}(\omega) )&=\iota_{\epsilon\delta\epsilon+e_{12}\delta e_{21}} (\mathrm{Tr}(\omega) )\\
&=\epsilon\iota_{\delta} (\mathrm{Tr}(\omega) )\epsilon+e_{12}\iota_{\delta } (\mathrm{Tr}(\omega) )e_{21}\\
&=\epsilon\iota_{\delta} (\epsilon\omega\epsilon+e_{12}\omega e_{21})\epsilon+e_{12}\iota_{\delta } (\epsilon\omega\epsilon+e_{12}\omega e_{21} )e_{21}\\
&=\epsilon\iota_{\delta} (\omega)\epsilon+e_{12}\iota_\delta(\omega) e_{21}\\
&=\mathrm{Tr} (\iota_{\delta} (\omega) )
\end{align*}
as wished.
\end{proof}

\subsection{Fusion and $1$-smoothness}

We start with the following notion simply called ``smoothness'' in~\cite{CBEG} or~\cite{VdB}.

\begin{definition}\label{1sm}We call  an $R$-algebra $A$ \emph{$1$-smooth} if it is finitely generated over $R$ and formally smooth in the sense of~\cite[\S19]{LecGinz}, meaning that $\Omega_{A/R}$ is a projective $A$-bimodule.\end{definition}

It implies that $A$ has projective dimension at most 1 and that we may (and will) use short resolutions.
Note that it implies smoothness of associated representation schemes, but we call it $1$-smooth in order to emphasize that it is way 
more demanding than the notion of (homological) smoothness we use in previous works~\cite{BCS,BCS2} for dg-categories (see also~\cref{cyrecap}), 
following e.g.~\cite{KellerDG}.

In the sequel, assume that $A=A_\cC$ where $\cC$ has a finite number of objects, and $R=\oplus_{e\in Ob(\cC)}ke$. 
\begin{proposition}
If $A$ is $1$-smooth over $R$, then so is $A^f$ over $R^f$.
\end{proposition}

\begin{proof}
Recall that ${A^+}=A \otimes_R {R^+}$. 
By definition $\Omega_{{A^+}/{R^+}} $ is the kernel of the multiplication map $m^+:{A^+} \otimes_{{R^+}} {A^+} \rightarrow {A^+} $ which can be identified with 
\[
\xymatrix{
{R^+} \otimes_R A \otimes_R A\otimes_R R^+\ar[rrr]^-{\id \otimes m\otimes\id }&&& {R^+} \otimes_R A\otimes_R R^+.}
\] 
Since $R$-modules are $Ob(\cC)\times Ob(\cC)$-graded $k$-vector space, $R^+$ is flat over $R$ and
\[
		  \Omega_{{A^+}/{R^+}}\simeq  (R^+)^e\otimes_{R^e} \Omega_{A/R}  
\simeq (R^+)^e\otimes_{R^e}A^e\otimes_{A^e} \Omega_{A/R}
\simeq (A^+)^e\otimes_{A^e} \Omega_{A/R}. 
\]
Since $\Omega_{A/R}$ is a projective $A$-bimodule,  $\Omega_{{A^+}/{R^+}} $  is a projective ${A^+}$-bimodule. 

Then, we know that $\Omega_{A^f/R^f}= e \Omega_{{A^+}/{R^+} }e$ from~\cref{lemma: trace-omega}. Since $ \Omega_{{A^+}/{R^+} } $ is a 
projective ${A^+} $-bimodule, there exists $r\in\mathbb N$ such that $\Omega_{A^f/R^f}$ is a direct summand of 
$e( {A^+} \otimes_{R^+} {A^+})^re=(e {A^+} \otimes_{R^+} {A^+}e)^r\simeq (A^f\otimes_{R^f} A^f)^r$ by~\cref{lemma: trace}. 
Hence $\Omega_{A^f/R^f}$ is a projective $A^f$-bimodule. 
\end{proof}


\section{Calabi--Yau versus bisymplectic structures}\label{section4}

In this section, we recall the notion of Calabi--Yau structures for dg-categories as in~\cite{BD1,ToEMS} and bisymplectic structures on algebras as in~\cite{CBEG}. 
We then introduce the fusion process for bisymplectic structures in analogy with the fusion for double Poisson structures from~\cite{VdB}. 
We show that a relative Calabi--Yau structure on $\coprod_{c\in\mathrm{Ob}(\cC)}k[x_c] \to \cC$, $\cC$ a $k$-linear category, gives rise to a bisymplectic one on 
the path algebra $A_\cC$ associated to $\cC$. Finally, we prove that the composition with the ``additive pair-of-pants'' Calabi--Yau cospan 
induces fusion for the corresponding bisymplectic structures on $A_\cC$. 

\subsection{Calabi--Yau structures, absolute and relative}\label{cyrecap}

Our notation follows \cite{BCS,BCS2}. 
A dg-category $\cA$ is called \emph{(homologically) smooth} if $\cA$ is a perfect $\cA^e$-module. 
In this case, we have an following equivalence 
\[
(-)^\flat: \cat{HH}( \cA) \stackrel{\sim} \longrightarrow    \mathbb{R}\ul{\cat{Hom}}_{\cat{Mod}_{\cA^e}} (\cA^\vee, \cA),
\]
where $\cA^\vee$ is the dualizing bimodule. 
\begin{definition}\label{def:CY}
Let $\cA$ be a smooth dg-category. An \emph{$n$-Calabi--Yau structure} on $\cA$ is a negative cyclic class $c=c_0+uc_1+\cdots:k[n]\to \cat{HC}^-(\cA)$ 
such that the underlying Hochschild class $c^\natural=c_0:k[n]\to \cat{HH}(\cA)$ is non-degenerate, in the sense that $c_0^\flat:\cA^\vee[n]\to \cA$ is an equivalence. 
\end{definition}

Relative Calabi--Yau structures on morphisms and cospans of dg-categories where introduced by Brav--Dyckerhoff~\cite{BD1} following To\"en~\cite[\S5.3]{ToEMS}. 

\begin{definition}
An \emph{$n$-Calabi--Yau structure} on a cospan $\cA\overset{f}{\longrightarrow}\cat{C}\overset{g}{\longleftarrow}\cat{B}$ of smooth dg-categories 
is a homotopy commuting diagram 
\[
\xymatrix{
k[n] \ar[r]^{c_{\cat{B}}}\ar[d]_{c_{\cA}}  & \cat{HC}^-(\cat{B}) \ar[d] \\
\cat{HC}^-(\cA) \ar[r] & \cat{HC}^-(\cat{C})
}
\]
whose image under $(-)^{\natural}$ is non-degenerate in the following sense: $c_{\cA}^\natural$ and $c_{\cat{B}}^\natural$ are non-degenerate, and 
the homotopy commuting square 
\[
\xymatrix{
\cat{C}^\vee[n]\ar[r]^-{g^\vee}\ar[d]_-{f^\vee} 
& (\cat{B}^\vee[n])\overset{\mathbb{L}}{\underset{\cat{B}^e}{\otimes}}\cat{C}^e 
\overset{(c_{\cat{B}}^\natural)^\flat\otimes\mathrm{id}}
{\simeq} \cat{B}\overset{\mathbb{L}}{\underset{\cat{B}^e}{\otimes}}\cat{C}^e \ar[d]^-{g\otimes\mathrm{id}} \\
(\cA^\vee[n])\overset{\mathbb{L}}{\underset{\cA^e}{\otimes}}\cat{C}^e 
\overset{(c_{\cA}^\natural)^\flat\otimes\mathrm{id}}{\simeq} \cA\overset{\mathbb{L}}{\underset{\cA^e}{\otimes}}\cat{C}^e \ar[r]^-{f\otimes\mathrm{id}}
& \cat{C}
}
\]
is cartesian. 
We say that a morphism $g: \cA \longrightarrow \cC$ is \emph{relative $n$-Calabi--Yau} if the copsan 
$\cA\overset{f}{\longrightarrow}\cat{C}{\longleftarrow}\varnothing$ is $n$-Calabi--Yau. 
\end{definition}
We will also use the fact that by~\cite[Theorem 6.2]{BD1} $n$-Calabi--Yau cospans compose. 
It is immediate with the above definitions that an $n$-Calabi--Yau structure on $\varnothing\rightarrow\cat{C}\leftarrow\varnothing$ is the same 
as an $(n+1)$-Calabi--Yau structure on $\cat{C}$. 
Finally recall (see e.g.~\cite[Proposition 2.3]{BCS2}) that a non-degenerate Hochschild class on a smooth dg-category $\cA$ concentrated in degree zero 
admits a unique cyclic lift, making $\cA$ an Calabi--Yau category. 

\begin{example}\label{exrcy}
\begin{itemize}
\item The algebra $k[x]$ carries a $1$-Calabi--Yau structure. We call the Calabi--Yau structure induced by $1\otimes x \in \cat{HH}_1(k[x])$ the natural Calabi--Yau structure. 
\item Let $Q=(I,E)$ be a finite quiver where $I$ is the set of vertices and $E$ the set of arrows. Denote by $\overline{Q}$ the double quiver obtained by adding for every arrow $a\in E$ an arrow $a^*$ in the opposite direction. Consider the path algebra of the double quiver $A:=k \overline{Q}$.
There is a relative $1$-Calabi--Yau structure on the moment map $k[x]\to kA$, $x\mapsto\sum_{a\in E}[a,a^*]$, which is compatible with the natural one on $k[x]$, see~\cite[5.3.2]{BCS}.
\item The algebra $k[x^{\pm1}]$ carries a natural $1$-Calabi--Yau structure induced by $\frac12(x^{-1}\otimes x-x\otimes x^{-1}) \in \cat{HH}_1(k[x^\pm])$. This has been shown in \cite{BCS2} Section 3.1. See also \cref{1cypm} for the cyclic lift.
\end{itemize} 
\end{example}

The next example of a Calabi--Yau cospan was investigated thoroughly in Section 3.3 of \cite{BCS2} and related to the pair-of-pants. 

\begin{example}[Pair-of-pants]\label{popex} 
 The cospan 
\begin{equation}\label{equation-multiplicative pants} k[x^{\pm1}] \amalg k[y^{\pm1}]  \longrightarrow  k\langle x^{\pm1} ,y^{{\pm1}} \rangle\longleftarrow k[z^{\pm1}]
\end{equation}   where the rightmost map is $z\mapsto xy$, is a relative $1$-Calabi--Yau cospan with the Calabi--Yau structures  $\alpha_1(x)+\alpha_1(y)-\alpha_1(z) =b(\beta_1)\sim 0$ and homotopy $\beta_1:=y^{-1}\otimes x^{-1}\otimes xy-y\otimes y^{-1}x^{-1}\otimes x$. 
\end{example}

 We prove here the additive version of the previous example which we refer to as the additive pair-of-pants as opposed to the multiplicative pair-of-pants of the previous example. 
 
\begin{lemma} 
There exists a relative $1$-Calabi--Yau structure on 
\begin{equation}\label{equation-additive pants}
k[x]\coprod k[y] \longrightarrow k\langle x,y\rangle  \longleftarrow k[z]\,,
\end{equation}
where the rightmost map is $z\mapsto x+y$, such that the underlying absolute $1$-Calabi--Yau structures on $k[x]$, $k[y]$ and $k[z]$ are the natural ones.
\end{lemma}

\begin{proof} 
 The algebra ${\mathcal B}:=k\langle x,y \rangle$ has a 
small resolution as a ${\mathcal B}$-bimodule: 
\[
({\mathcal B}^e)^{\oplus2}[1] \oplus {\mathcal B}^e
\]
with differential sending $(1\otimes1,0)$ to $x\otimes 1-1\otimes x$, and $(0,1\otimes1)$ to $y\otimes 1-1\otimes y$. 
Therefore 
\[
{\mathcal B}^\vee\simeq {\mathcal B}^e\oplus ({\mathcal B}^e)^{\oplus2}[-1]
\]
with differential sending $1\otimes 1$ to $(x\otimes 1-1\otimes x,y\otimes 1-1\otimes y)$. 

The canonical Calabi--Yau structures on $\mathcal A:=k[x]$ are given by $\alpha_1(x)= 1\otimes x \in \mathrm{HH}_1(\mathcal{A})$. Note that $\alpha_1$ has a unique cyclic lift by Proposition 2.3 of \cite{BCS2} which we denote $\alpha$. 
The following diagram induced by the natural Calabi--Yau structures on $\mathcal A$ is
strictly commutative
\[
\xymatrix{ {\mathcal B}^\vee[1] \ar[d]\ar[r] &  {\mathcal A}^\vee \underset{{\mathcal A}^e}{\otimes}{\mathcal B}^e[1] \overset{\alpha_1(x+y)}{\simeq}  {\mathcal A}\underset{{\mathcal A}^e}{\otimes}{\mathcal B}^e \ar[d]  \\
( {\mathcal A}^{\oplus2})^\vee \underset{{\mathcal A}^e}{\otimes} {\mathcal B}^e [1] \overset{\alpha_1(x) +\alpha_1(y)}{\simeq} {\mathcal A}^{\oplus 2} \underset{{\mathcal A}^e}{\otimes}{\mathcal B}^e\ar[r] & {\mathcal B}
}
\]
Using the small resolution of $\mathcal{A}$ we find ${\mathcal A}\underset{{\mathcal A}^e}{\otimes}{\mathcal B}^e \simeq {\mathcal B}^e[1] \oplus {\mathcal B}^e$, with differential sending $1\otimes1$ to $x \otimes 1 -1\otimes x$. 
Hence, we get that the diagram is cartesian. The zero homotopy is the unique lift in cyclic homology between $\alpha(z)$ and $\alpha(x)+\alpha(y)$. 
Therefore the cospan~\eqref{equation-additive pants} carries a relative $1$-Calabi--Yau structure. 
\end{proof}

\subsection{Bisymplectic structures and fusion}

Let $A$ be an $R$-algebra, where $R= ke_1 \oplus \cdots \oplus ke_n$ is based on pairwise orthogonal idempotents as usual. 
We define gauge elements $E_i=(a\mapsto ae_i\otimes e_i-e_i\otimes e_ia)\in D_{A/R}$ and recall notions introduced in~\cite{CBEG}.

\begin{definition}
 We call   $\omega \in \Omega^2_R(A)$ a \emph{bisymplectic} structure on $A$, if 
  \begin{itemize}
 \item $\omega $ is {closed}, that is $d \omega=0 \in \DR_R(A)$ 
 \item $\omega$ is{ non-degenerate}, that is $\iota(\omega): D_{A/R} \to \Omega_{A/R}, \delta\mapsto\iota_\delta(\omega)$ is an isomorphism. 
 \end{itemize} 
An element $\mu=(\mu_i)\in\oplus_ie_iAe_i$ is a \emph{moment map} for a bisymplectic algebra $(A,\omega)$ if\[
d\mu_i=\iota_{E_i}(\omega)\]
for all $i\in I$.
 \end{definition} 
 
A moment map always exists, see~\cite[A.7]{VdB}.
Now we discuss fusion of bisymplectic structures and aim to prove~\cite[Proposition 2.6.6]{VdB}.
We use the notation of \cref{section: fusion}. Recall that  we have trace maps $A \to A^f, a \mapsto  a^f=\epsilon a \epsilon+ e_{12} a e_{21}$,  
$\Omega_R^* (A)\to \Omega_{R^f}^* (A^f)$ and $D_R^*(A) \to D_{R^f}^*(A^f)$.
Let $A$ be an algebra equipped with a bisymplectic structure $\omega$, with moment map $\mu$. We define $\mu_i^{f\!\!f}=\mu_i^f=\mu_i $ 
for $i\ge 3$ and \[\mu_1^{f\!\!f}= \mu_1 +e_{12} \mu_2 e_{21} = \mu_1^f+ \mu_2^f.\]

\begin{lemma}
The form $\omega^f\in\Omega_{R^f}^2 (A^f)$ is a bisymplectic structure on $A^f$, with moment map $\mu^{f\!\!f}$. 
\end{lemma}

\begin{proof} 
By definition, $\omega^f \in \Omega^2_{R^f}A^f$ is a closed form. 
We need to show that $\iota(\omega^f): D_{A^f/R^f} \to \Omega_{A^f/R^f}$ is an isomorphism. Recall from \cref{lemma: trace functoriality}, that 
we have the following commutative diagram 
\[
 \xymatrix{ D_{A/R} \ar[r] \ar[d]^{\iota(\omega)} & D_{{A^+}/{R^+}} \ar[r] \ar[d]^{\iota({\omega^+}) } &  D_{A^f/R^f} \ar[d]^{ \iota(\omega^f) } \\
 \Omega_{A/R}  \ar[r]&  \Omega_{{A^+}/ {R^+}}\ar[r]&  \Omega_{A^f/ R^f}.}
 \]
 Now $\iota({\omega^+})$ is an isomorphism as it is obtained from $\iota(\omega)$ by an extension of rings $-\otimes_R {R^+}$, where $R$ is semi-simple. 

We observe that the map $\mathrm{Tr}: \Omega_{{A^+}/{R^+}} \to \Omega_{A^f/R^f}$ is surjective. As $\iota({\omega^+})$ is surjective,
 $\iota(\omega^f)$ is also surjective by \cref{lemma: trace functoriality}. Furthermore, the kernel of $\mathrm{Tr}: \Omega_{{A^+}/{R^+}} \to \Omega_{A^f/R^f}$ 
 is given by $ \epsilon \Omega_{A/R} e_2 + e_2 \Omega_{A/R} \epsilon$ and the kernel of $\mathrm{Tr}: D_{{A^+}/{R^+}} \to D_{A^f/R^f}$ is 
 $\epsilon D_{{A^+}/ {R^+}} e_2+ e_2D_{{A^+}/ {R^+}} \epsilon$. The morphism $\iota({\omega^+})$ maps the two kernels bijectively to each other as 
 it is an ${A^+} \otimes_{{R^+} } {A^+}$-linear isomorphism. Furthermore, $\mathrm{Tr}: D_{{A^+}/ {R^+}} \to D_{A^f/R^f}$ is surjective. As a consequence, 
 $\iota(\omega^f)$ is also an isomorphism proving that $\omega^f$ is non-degenerate. 
 This shows that $\omega^f$ is a bisymplectic structure o, $A^f$. The moment map $\mu:=( \mu_i )_i$ associated to $\omega$ is determined by the condition  $d \mu_i =\iota_{E_i} (\omega)$. Denote by $F_i$ for $i\neq 2$ the gauge elements in $A^f$. By \cref{lemma: trace functoriality}  
\[d (\mu^f_i)=(d\mu_i)^f =( \iota_{E_i} (\omega))^f= \iota_{E_i^f} (\omega^f)= \iota_{F_i} (\omega^f)\]
 for $ i \not= 1,2$. We know from~\cite[Lemma 5.3.3]{VdB} that $F_1=E_1^f+E_2^f$, so
\[d(\mu_1^{f\!\!f})=d(\mu_1^f+\mu_2^f)= \iota_{E_1^f} (\omega^f) + \iota_{E_2^f} (\omega^f) = \iota_{F_1} (\omega^f)\]
as expected.
\end{proof} 

\subsection{From Calabi--Yau structures to bisymplectic structures}

Let $\mathcal C$ be a $k$linear category with set of objects $I=\{1, \dots , n\}$ (in particular, we assume that $\cC$ is concentrated in degree $0$). 
Set $e_i=\mathrm{id}_i$, $R= \oplus_{i\in I} ke_i$, 
$\hat{\mathcal R}= \coprod_{i\in I} k[x_i]$ and $A=A_\cC$. 
Note that $\hat{R}:=A_{\hat{\mathcal R}}\simeq \bigoplus_{i\in I} k[x_i]$. 
We assume that we are given an endomorphism of each object $i$. 
This amounts to having a $k$-linear functor $\mu: \hat{\mathcal R} \to \mathcal C$ or, equivalently, an $R$-algebra morphism $\hat{\mathcal R}\to A$. 
Let us set $\mu_i:=\mu(x_i) \in e_iA e_i$.

\begin{theorem}\label{rcybisym}
Assume we have a relative $1$-Calabi--Yau structure on  $\mu: \hat{\mathcal R} \to \cC$ inducing the natural Calabi--Yau structure on each $k[x_i]$ and assume that $A_\cC$ is $1$-smooth. Then $A_{\cC}$ is bisymplectic with moment map $ \sum_{i=1}^n \mu_i$. 
\end{theorem} 
\begin{proof} 
 The $1$-Calabi--Yau structure gives a homotopy $0 \sim \mu( \sum_{i=1}^n 1\otimes x_i) = \sum_{i=1}^n 1\otimes \mu_i$ which yields, thanks to~\cref{subsechhncform}, an element $\omega_1 \in  \Omega_R^2(A)$ satisfying $ \iota_{E} (\omega_1)= \sum_{i=1}^n d\mu_i$. Hence $ \mu$ is a moment map for $\omega_1$. 

It remains to show that $\omega_1$  is closed and non-degenerate. 
First note that $\gamma:=\sum_{i=1}^n 1\otimes x_i\in\Omega_R^1\hat{R}$ trivially lifts in negative cyclic homology as $B(\gamma)=0$. Then the Calabi--Yau structure is given by a family $\omega_k  \in  \bar{\Omega}_R^{2k} A$ satisfying\[
(\iota_E-ud )\left(\sum_{k\ge0} u^k\omega_{k+1}\right)= \mu( \gamma),\]
which implies $ d\omega_1=\iota_E (\omega_2)=0\in \overline\DR_RA$.
This proves the closedness of $\omega_1$. 

The (Calabi--Yau) non-degeneration property yields the homotopy fiber sequence
 \[
A^\vee[1]\to R^\vee[1]\otimes_{R^e}A^e\overset{\gamma_1}{\simeq} R\otimes_{R^e}A^e\to A. 
\]
Using short resolutions (thanks to the $1$-smoothness of $A$) we get the homotopy commuting diagram\[
\xymatrix{A^e\ar[r]^-{\mathrm{id}}\ar[d]_-{E}&A^e\ar[rrr]^-{\id}\ar[d]^{\mu\otimes \id - \id \otimes \mu }&&&A^e\ar[r]^-{d\mu}\ar[d]_{\mu\otimes \id - \id \otimes \mu }&\Omega_{A/R}\ar@{_(->}[d]\\
D_{A/R}\ar[r]_-{\mathrm{ev}_\mu}&A^e\ar[rrr]_-{\id }&&&A^e\ar[r]_-{\mathrm{id}}&A^e}\] 
The homotopy is given by $\iota(\omega_1): D_{A/R}\to\Omega_{A/R}$ 

Now as the Calabi--Yau structure is non-degenerate we have \[
A^\vee[1]\simeq\mathrm{hofib}\left( R^\vee[1]\otimes_{R^e}A^e\overset{\gamma_1}{\simeq} R\otimes_{R^e}A^e\to A\right).\]
In short resolutions, this yields a quasi-isomorphism between the vertical complexes
\[
\xymatrix{A^e\ar[r]^-{ \id}\ar[d]_-{E}&A^e\ar[d]^-{d\mu}\\
D_{A/R}\ar[r]_-{\iota(\omega_1)}&\Omega_{A/R}}\]
which in particular gives an isomorphism $\iota(\omega_1): D_{A/R} \to \Omega_{A/R}.$ 
\end{proof}

\begin{example}\label{exqcbeg}
Let $Q=(I,E)$ be a finite quiver where $I$ is the set of vertices and $E$ the set of arrows. Denote by $\overline{Q}$ the double quiver obtained by adding for every arrow $a\in E$ an arrow $a^*$ in the opposite direction. Consider the path algebra of the double quiver $A:=k \overline{Q}$. We have
\begin{itemize}
\item a relative $1$-Calabi--Yau structure on $ \mu: k[x] \to A$, $x\mapsto\sum_{a\in E}[a,a^*]$ from \cref{exrcy};
\item a  bisymplectic structure $\omega=\sum_{a\in E}dada^* \in \overline{\mathrm{DR}}_R^2A$ on $A$ given in~\cite[Proposition 8.1.1]{CBEG}, with moment map $\mu$.
\end{itemize} 
We claim that the first structure implies (twice) the second one under \cref{rcybisym}.
Indeed: the homotopy between $0$ and $\mu(1\otimes x)$ is given by $\sum_{a\in E}(1\otimes a\otimes a^*-1\otimes a^*\otimes a)$ which  corresponds to $2\sum_{a\in E}dada^*$.
\end{example}

We next investigate the relationship between fusion of bisymplectic structures and relate them to the compositions of Calabi--Yau cospans. Consider a dg-category $\cC$ with objects set $I$, along with a relative $1$-Calabi--Yau structure $\mu: \hat{\mathcal R}\to \cC$ that induces natural absolute Calabi--Yau structures on each $k[x_i]$.
Set $\hat{\mathcal R}_{\ge3}= \coprod_{i\ge3} k[x_i]$.
We can consider the composition of cospans
\[
\xymatrix{&& \mathcal C^f   &&\\
& k\langle x_1, x_2 \rangle \amalg\hat{\mathcal R}_{\ge3}  \ar[ur]&&~~~~~\cC~~~~~\ar[ul]&\\
k[z] \amalg\hat{\mathcal R}_{\ge3}   \ar[ur]&&  \hat{\mathcal R}   \ar[ul]\ar[ur]&&~~~~\varnothing~~~~\ar[lu]}\]
defining $\mathcal C^f$, where $z$ is mapped to $x_1+x_2$. 
This yields a relative Calabi--Yau structure on \begin{equation}\label{fusadd}
k[z] \amalg\hat{\mathcal R}_{\ge3}  \to \mathcal C^f.\end{equation}

\begin{theorem}\label{fusthmadd}
Assume that $A_\cC$ is $1$-smooth. 
Let $(A_{\cC}, \omega)$  be the bisymplectic structure induced by the relative 1-Calabi--Yau structure $\mu$ thanks to~\cref{rcybisym}. 
Then the fusion bisymplectic structure $(A_{\cC}^f, \omega^f)$  obtained from fusing the two objects $1$ and $2$ is induced by the relative $1$-Calabi--Yau structure~\eqref{fusadd}. 
\end{theorem}

\begin{proof}
Set $A=A_{\cC}$. We know thanks to \cref{prop: fusion} that $A^f \simeq A_{ \mathcal C^f}$. As the bisymplectic structure is compatible with the relative 1-Calabi--Yau structure, we have that the image of $z$ under this isomorphism is $\mu(x_1)^f+\mu(x_2)^f$. Hence the moment map of the fusion bisymplectic structure is induced from the Calabi--Yau cospan. 
Let $\omega \in \Omega_R^2( A)$ denote the homotopy $ \mu(1\otimes (\sum_{i\in I} x_i)) \sim 0$ of the Calabi--Yau structure which induces by assumption  the bisymplectic structure on $A$. 
Since the homotopy between the 1-forms in the cospan \[
 k[z] \amalg\hat{\mathcal R}_{\ge3}   \longrightarrow   k\langle x_1, x_2 \rangle \amalg\hat{\mathcal R}_{\ge3} \longleftarrow \hat{\mathcal R} \]
  is trivial, the zero-homotopy of  the composition of Calabi--Yau cospans is given by the image of $\omega $ under 
 the map $\nu$ from \cref{lemma: HH and fusion}. But it is proven there that this image is $\omega^f$, which is precisely what we want.
\end{proof} 

To summarize, we have proven that the following diagram commutes, with 
$R^f\simeq \oplus_{i\in I\setminus\{2\}}ke_i$ and $\hat{\mathcal R}^f\simeq\amalg_{i\in I\setminus\{2\}}k[x_i]$.\[
\xymatrix{
{\left\{\begin{array}{c}1$-Calabi--Yau functors$ \\ 
\hat{\mathcal R}\to\cC$, under $\mathcal{R}, \\
$with $A_\cC$ $1$-smooth$\end{array}\right\}}
\ar[rrr]^-{\text{Theorem }\ref{rcybisym}}\ar[dd]_-{\substack{\text{composition} \\ 
\text{with pair-of-pants}}}&&&{\left\{\begin{array}{c}$bisymplectic structures$\\$on $1$-smooth $R$-algebras$\end{array}\right\}}\ar[dd]^-{\text{fusion}}\\&&&\\
{\left\{\begin{array}{c}1$-Calabi--Yau functors$ \\ \hat{\mathcal R}^f\to\cC^f$, under $\mathcal{R}^f,\\
$with $A_{\cC^f}$ $1$-smooth$\end{array}\right\}}\ar[rrr]^-{\text{Theorem }\ref{rcybisym}}&&&{\left\{\begin{array}{c}$bisymplectic structures$\\$on $1$-smooth $R^f$-algebras$\end{array}\right\}}
}\]


\section{Calabi--Yau versus quasi-bisymplectic structures}\label{secvyqb}

We prove in this section that  relative Calabi--Yau structures on $k[x^{\pm1}] \to \cC,$ $\cC$ a $k$-linear dg-category, induces this time quasi-bisymplectic ones on $A_\cC$, in the sense of~\cite{VdB2}. We prove again that  fusion of quasi-bisymplectic structures on  $A_\cC$ is induced by the composition of Calabi--Yau cospans with the multiplicative pair-of-pants.

\subsection{Quasi-bisymplectic structures}

Consider an $R$-algebra $A$.

\begin{definition}[\cite{VdB2}] A \emph{quasi-bisymplectic algebra} is a triple
$(A,\omega,\Phi)$ where $\omega\in \cat{DR}_R^2 A$ and $\Phi\in A^\ast$ satisfying the following conditions
\begin{enumerate}
\item[($\mathbb{B}$1)] $d\omega=\frac{1}{6} (\Phi^{-1} d\Phi)^3\quad \mod [-,-]$.
\item[($\mathbb{B}$2)] $\imath_{E}\omega=\frac{1}{2} (\Phi^{-1}d\Phi+d\Phi \Phi^{-1})$
\item[($\mathbb{B}$3)] The map
\[
D_{A/R}\oplus Ad\Phi A\rightarrow \Omega_A :(\delta,\eta)\mapsto
\imath(\omega)(\delta)+\eta
\]
is surjective.
\end{enumerate}
\end{definition}

Recall from~\cite[Theorem 7.1]{VdB2} the $A\otimes_RA$-linear map $T:\Omega_{A/R} \stackrel{e} \rightarrow A E^* A \stackrel{T^0}  \rightarrow A d\Phi A  \stackrel{c} \rightarrow \Omega_{A/R}$, where $c$ denotes the canonical embedding, $e$ the adjoint of $c$ and $T^0$ is uniquely determined by $T^0(E^*)=\Phi^{-1} d\Phi-d\Phi \Phi^{-1}$.
\begin{definition}
We say that a triple $(\omega,P, \Phi)\in\Omega_R^2(A)\times D_R^2(A)\times A^*$ is \emph{compatible}  if $\iota(\omega) \iota(P)=1-\frac{1}{4}T$.
\end{definition}

What is proved by~\cite[Theorem 7.1]{VdB2} is that each quasi-bisymplectic structure of $\cat{DR}_R^2 (A)$ corresponds to a unique non-degenerate double quasi-Poisson bracket in $(D_RA/[D_RA,D_RA])_2$. We will not recall the definition of the latter here.

\begin{lemma}
Let $(\omega, P,\Phi)$ be a compatible triple on $A$ such that $(\omega,\Phi)$ is quasi-bisymplectic. Then $({\omega^+},{\Phi^+})$ is quasi-bisymplectic on $A^+$ and  $({\omega^+},  P^+,{\Phi^+})$ is also compatible.  
\end{lemma}

\begin{proof}
The compatibility condition is given by $\iota(\omega) \iota(P)=1-\frac{1}{4}T$. 
Since $R$ is semi-simple,  $-\otimes_R {R^+}$ is exact. Recall also that $\Omega_{{A^+}/{R^+}} \simeq \Omega_{A/R} \otimes_R {R^+}$ and $D_{{A^+}/{R^+}} \simeq  D_{A/R} \otimes_R {R^+}$. From this it follows immediately that $({\omega^+},  {\Phi^+})$ is a quasi-bisymplectic structure. 
Now by functoriality of the extension of scalar functor $-\otimes_R {R^+}$, we obtain that 
$\iota( \omega^+) \iota(  P^+)=1-\frac{1}{4}  T^+$. 
\end{proof} 

Assume that $R=\oplus_{i\in I}ke_i$ is based on pairwise orthogonal idempotents. Let $(\omega, P,\Phi)$ be a compatible triple on $A$ such that $(\omega,\Phi)$ is quasi-bisymplectic and assume that $\Phi=(\Phi_i)_{i\in I}\in\oplus_{i\in I}e_i A^\ast e_i$.
 Set $\Phi_1^{f\!\!f}=\Phi_1^f\Phi_2^f$ and $\Phi_i^{f\!\!f}=\Phi_i^f=\Phi_i$ if $i>2$.
The following rather computational result is the noncommutative analog of~\cite[Proposition 10.7]{AKSM}.
\begin{proposition}\label{fusionqH}
Set $\omega_\mathrm{cor}=\frac{1}{2}(\Phi_1^f)^{-1}d\Phi_1^fd\Phi_2^f(\Phi_2^f)^{-1}$. Then $\omega^{f\!\!f}:=\omega^f-\omega_\mathrm{cor}$ is compatible with $P^{f\!\!f}:=P^f+\frac{1}{2}E_1^fE_2^f$.
\end{proposition}

\begin{proof}
We need to prove that $\iota(\omega^{f\!\!f})\iota(P^{f\!\!f})=1-\frac{1}{4}T^{f\!\!f}$ which is equivalent to\begin{equation}\label{compafusion}
\underbrace{\iota(\omega^f)\iota(P^f)}_{\mytag{(I)}{termI}}-\frac{1}{2}\underbrace{\iota(\omega_\mathrm{cor})\iota(E_1^fE_2^f)}_{\mytag{(II)}{termII}}-\underbrace{\iota(\omega_\mathrm{cor})\iota(P^f)}_{\mytag{(III)}{termIII}}+\frac{1}{2}\underbrace{\iota(\omega^f)\iota(E_1^fE_2^f)}_{\mytag{(IV)}{termIV}}=1-\frac{1}{4}\underbrace{T^{f\!\!f}}_{\mytag{(V)}{termV}}.\end{equation}
Note that  $ A^+ \to A^f$, $a \mapsto \mathrm{Tr}(a)$ is surjective. Hence it is sufficient to show compatibility on all images of  $da\in  \Omega_{{A^+}/{R^+}}$. We will systematically use the notation $(-)^f=\mathrm{Tr}(-)$ in the rest of this proof.

We have $\Phi_1^{f\!\!f}=\Phi_1^f\Phi_2^f=\Phi_1^+e_{12}\Phi_2^+e_{21}$ and $\Phi_i^{f\!\!f}=\Phi_i^f=\Phi_i$ if $i>2$. We abusively note $\Phi_i=\Phi_i^+$ as they don't involve $e_{ij}$'s, so that $\Phi_i^f=\Phi_i$ when $i\neq 2$, $\Phi_2^f=e_{12}\Phi_2e_{21}$ and we set $\Psi=\Phi_1^{f\!\!f}$.
Then for any $a\in A^+$ \begin{align*}
\text{\ref{termV}}(da^f)&=T^{f\!\!f}(da^f)\\
&=[a^f,(\Phi^{f\!\!f})^{-1}d\Phi^{f\!\!f}-d\Phi^{f\!\!f}(\Phi^{f\!\!f})^{-1}]\\
&=[a^f,\Psi^{-1}d\Phi_1\Phi_2^f+(\Phi_2^f)^{-1}d\Phi_2^f-d\Phi_1\Phi_1^{-1}-\Phi_1d\Phi_2^f\Psi^{-1}]\\
&\qquad\qquad +\sum_{i>2}[a^f,\Phi_i^{-1}d\Phi_i-d\Phi_i\Phi_i^{-1}]\\
&=[a^f,\Psi^{-1}d\Phi_1\Phi_2^f+(\Phi_2^f)^{-1}d\Phi_2^f-d\Phi_1\Phi_1^{-1}-\Phi_1d\Phi_2^f\Psi^{-1}]\\
&\qquad\qquad +\sum_{i>2}\epsilon[a,\Phi_i^{-1}d\Phi_i-d\Phi_i\Phi_i^{-1}]\epsilon
\end{align*}
whereas, thanks to \cref{lemma: trace functoriality},\allowdisplaybreaks \begin{align*}
\text{\ref{termI}}(da^f)&=\iota(\omega^f)\iota(P^f)(da^f)\\
&=\iota(\omega^f) \big(\iota(P)(da)\big)^{\!f} \\
&= \big( \iota(\omega) \iota( P) (da)\big)^{\!f}\\
&= \big(a-\frac{1}{4}T(da)\big)^{\!f} \\
&=a^f-\dfrac{1}{4}\epsilon[a,\Phi_1^{-1}d\Phi_1-d\Phi_1\Phi_1^{-1}]\epsilon-\dfrac{1}{4}e_{12}[a,\Phi_2^{-1}d\Phi_2-d\Phi_2\Phi_2^{-1}]e_{21}\\
&\qquad\qquad-\frac{1}{4}\sum_{i>2}\epsilon[a,\Phi_i^{-1}d\Phi_i-d\Phi_i\Phi_i^{-1}]\epsilon\\
&=a^f-\dfrac{1}{4}[\epsilon a\epsilon,\Phi_1^{-1}d\Phi_1-d\Phi_1\Phi_1^{-1}]\epsilon-\dfrac{1}{4}[e_{12}ae_{21},(\Phi_2^f)^{-1}d\Phi_2^f-d\Phi_2^f(\Phi_2^f)^{-1}]\\
&\qquad\qquad-\frac{1}{4}\sum_{i>2}\epsilon[a,\Phi_i^{-1}d\Phi_i-d\Phi_i\Phi_i^{-1}]\epsilon.\end{align*}
Recall that for every $\delta\in D_{A^f}$\begin{align*}
2\iota(\omega_\mathrm{cor})(\delta)&={}^\circ i_{\delta}(\Phi_1^{-1}d\Phi_1d\Phi^f_2(\Phi_2^f)^{-1})\\
&={}^\circ(\Phi_1^{-1}\delta\Phi_1d\Phi^f_2(\Phi_2^f)^{-1}-\Phi_1^{-1}d\Phi_1\delta\Phi^f_2(\Phi_2^f)^{-1})\\
&=\delta(\Phi_1)''d\Phi^f_2\Psi^{-1}\delta(\Phi_1)'-\delta(\Phi_2^f)''\Psi^{-1}d\Phi_1\delta(\Phi^f_2)'.
\end{align*}
and that for every $a\in A$ we have $\iota(P)(da)=H_a$, the Hamiltonian vector field which satisfies \[
H_a(\Phi)=-\dfrac{1}{2}(\Phi E+E\Phi)(a)^\circ.\]
It implies (recall that the bimodule structure on double derivations is induced by the inner one on $A\otimes_RA$)\begin{align*}
2H^f_a(\Phi_1^f)&=2(\epsilon H_a\epsilon+e_{12}H_ae_{21})(\Phi_1)\\
&=-(\epsilon\Phi_1E_1\epsilon+\epsilon E_1\Phi_1\epsilon+e_{12}\Phi_1E_1e_{21}+e_{12} E_1\Phi_1e_{21})(a)^\circ\\
&=-(\Phi_1E_1\epsilon+\epsilon E_1\Phi_1)(a)^\circ\\
&=-(a\epsilon\otimes \Phi_1-\epsilon\otimes \Phi_1a+a\Phi_1\otimes \epsilon-\Phi_1\otimes \epsilon a)^\circ\\
&=-\Phi_1\otimes a\epsilon+\Phi_1a\otimes \epsilon-\epsilon\otimes a\Phi_1+\epsilon a\otimes \Phi_1\\
&=-\Phi_1\otimes \epsilon a \epsilon+\Phi_1\epsilon a \epsilon\otimes e_1-e_1\otimes \epsilon a \epsilon\Phi_1+\epsilon a \epsilon\otimes \Phi_1
\end{align*}
and\begin{align*}
2H^f_a(\Phi_2^f)&=2e_{12}(\epsilon H_a\epsilon+e_{12}H_ae_{21})(\Phi_2)e_{21}\\
&=-\big(e_{12}(\epsilon\Phi_2E_2\epsilon+\epsilon E_2\Phi_2\epsilon+e_{12}\Phi_2E_2e_{21}+e_{12} E_2\Phi_2e_{21})(a)e_{21}\big)^{\!\circ}\\
&=-\big(e_{12}(e_{12}\Phi_2E_2e_{21}+e_{12} E_2\Phi_2e_{21})(a)e_{21}\big)^{\!\circ}\\
&=-\big(e_{12}ae_{21}\otimes e_{12}\Phi_2e_{21}-e_{12}e_{21}\otimes e_{12} \Phi_2ae_{21}\\
&\qquad\qquad+e_{12}a\Phi_2e_{21}\otimes e_{12}e_2e_{21}-e_{12}\Phi_2e_{21}\otimes e_{12}e_2ae_{21}\big)^{\!\circ}\\
&=-\big(e_{12}ae_{21}\otimes \Phi_2^f-e_{1}\otimes\Phi_2^fe_{12}ae_{21}+e_{12}ae_{21}\Phi_2^f\otimes e_{1}-\Phi_2^f\otimes e_{12}ae_{21}\big)^{\!\circ}\\
&=-\Phi_2^f\otimes e_{12}ae_{21}+\Phi_2^fe_{12}ae_{21}\otimes e_{1}-e_1\otimes e_{12}ae_{21}\Phi_2^f+e_{12}ae_{21}\otimes\Phi_2^f.
\end{align*}
We thus obtain \begin{align*}		
4\text{\ref{termIII}}(da^f)&=4\iota(\omega_\mathrm{cor})\iota(P^f)(da^f)\\
&=4\iota(\omega_\mathrm{cor})(H_a^f)\\
&=2H^f_a(\Phi_1)''d\Phi_2^f\Psi^{-1}H^f_a(\Phi_1)'-2H^f_a(\Phi_2^f)''\Psi^{-1}d\Phi_1H^f_a(\Phi_2^f)'\\
&=-e_1ad\Phi_2^f\Psi^{-1}\Phi_1+d\Phi_2^f\Psi^{-1}\Phi_1ae_1-e_1a\Phi_1d\Phi_2^f\Psi^{-1}+\Phi_1d\Phi_2^f\Psi^{-1}ae_1\\
&+e_{12}ae_{21}\Psi^{-1}d\Phi_1\Phi_2^f-\Psi^{-1}d\Phi_1\Phi_2^fe_{12}ae_{21}+e_{12}ae_{21}\Phi_2^f\Psi^{-1}d\Phi_1-\Phi_2^f\Psi^{-1}d\Phi_1e_{12}ae_{21}\\
&=-\epsilon a \epsilon d\Phi_2^f(\Phi_2^f)^{-1}+d\Phi_2^f(\Phi_2^f)^{-1}\epsilon a \epsilon-\epsilon a \epsilon\Phi_1d\Phi_2^f\Psi^{-1}+\Phi_1d\Phi_2^f\Psi^{-1}\epsilon a \epsilon\\
&+e_{12}ae_{21}\Psi^{-1}d\Phi_1\Phi_2^f-\Psi^{-1}d\Phi_1\Phi_2^fe_{12}ae_{21}+e_{12}ae_{21}\Phi_1^{-1}d\Phi_1-\Phi_1^{-1}d\Phi_1e_{12}ae_{21}\\
&=-[\epsilon a \epsilon,d\Phi_2^f(\Phi_2^f)^{-1}+\Phi_1d\Phi_2^f\Psi^{-1}]+[e_{12}ae_{21},\Psi^{-1}d\Phi_1\Phi_2^f+\Phi_1^{-1}d\Phi_1].
\end{align*}
Also\begin{align*}
2\iota(\omega_\mathrm{cor})\iota(E_1^fE_2^f)(da^f)&=2\iota(\omega_\mathrm{cor}){}^\circ(i_{da^f}(E_1^f)E_2^f-E_1^fi_{da^f}(E_2^f))\\
&=2\iota(\omega_\mathrm{cor}){}^\circ(E_1^f(a^f)E_2^f-E_1^fE_2^f(a^f))\\
&=2\iota(\omega_\mathrm{cor})(e_1E_2^f\epsilon ae_1-e_1a\epsilon E_2^fe_1-e_{1}E_1^fe_{12}ae_{21}+e_{12}ae_{21}E_1^fe_{1})\\
&=e_1\iota(2\omega_\mathrm{cor})(E_2^f)\epsilon ae_1-e_1a\epsilon\iota(2\omega_\mathrm{cor})(E_2^f)e_1\\
&\qquad-e_1{\iota(2\omega_\mathrm{cor})({E_1^f})}e_{12}ae_{21}+e_{12}ae_{21}\iota(2\omega_\mathrm{cor})(E_1^f)e_{1}.
\end{align*}
But\begin{align*}
E_1^f(a^f)&=\epsilon E_1^+(a)\epsilon+e_{12} E_1^+(a)e_{21}\\
&=\epsilon ae_1\otimes e_1\epsilon-\epsilon e_1\otimes e_1a\epsilon+e_{12}ae_1\otimes e_1e_{21}-e_{12}e_1\otimes e_1ae_{21}\\
&=\epsilon a\epsilon \otimes e_1- e_1\otimes \epsilon a\epsilon
\end{align*}
and\begin{align*}
E_2^f(a^f)&=\epsilon(e_{12} E_2^+e_{21})(a)\epsilon+e_{12}(e_{12} E_2^+e_{21})(a)e_{21}\\
&=\epsilon ae_{21}\otimes e_{12}\epsilon-\epsilon e_{21}\otimes e_{12}a\epsilon+e_{12}ae_{21}\otimes e_{12}e_{21}-e_{12}e_{21}\otimes e_{12}ae_{21}\\
&=e_{12}ae_{21}\otimes e_{1}-e_{1}\otimes e_{12}ae_{21}
\end{align*}
imply $E_1^f(\Phi_1)=E_1(\Phi_1)$, $E_1^f(\Phi_2^f)=0$, $E_2^f(\Phi_1)=0$, $E_2^f(\Phi_2^f)=E_1(\Phi_2^f)$ and \begin{align*}
\iota(2\omega_\mathrm{cor})(E_1^f)&=d\Phi^f_2(\Phi_2^f)^{-1}-\Phi_1d\Phi^f_2\Psi^{-1}\\
\iota(2\omega_\mathrm{cor})(E_2^f)&=-\Psi^{-1}d\Phi_1\Phi^f_2+\Phi_1^{-1}d\Phi_1.
\end{align*}
Hence,\begin{align*}
2\text{\ref{termII}}(da^f)&=2\iota(\omega_\mathrm{cor})\iota(E_1^fE_2^f)(da^f)\\
&=e_1(-\Psi^{-1}d\Phi_1\Phi^f_2+\Phi_1^{-1}d\Phi_1)\epsilon a\epsilon-\epsilon a\epsilon(-\Psi^{-1}d\Phi_1\Phi^f_2+\Phi_1^{-1}d\Phi_1)e_1\\
&-e_1(d\Phi^f_2(\Phi_2^f)^{-1}-\Phi_1d\Phi^f_2\Psi^{-1})e_{12}ae_{21}+e_{12}ae_{21}(d\Phi^f_2(\Phi_2^f)^{-1}-\Phi_1d\Phi^f_2\Psi^{-1})e_{1}\\
& =[e_{12} a e_{21}, d\Phi^f_2 (\Phi^{f}_2)^{-1}- \Phi_1d\Phi_2^f \Psi^{-1}] +[\epsilon a \epsilon, \Psi^{-1}d\Phi_1  \Phi^{f}_2-\Phi_1^{-1}d \Phi_1 ].
\end{align*}
Similarly, using $\iota(2\omega^f)(E_i^f)=(\Phi_i^{-1}d\Phi_i+d\Phi_i\Phi_i^{-1})^f$, one gets\begin{align*}
2\text{\ref{termIV}}(da^f)&=2\iota(\omega^f)\iota(E_1^fE_2^f)(da^f)\\
&=e_1\iota(2\omega^f)(E_2^f)\epsilon a\epsilon- \epsilon a\epsilon\iota(2\omega^f)(E_2^f)e_1\\
&\qquad-e_1{\iota(2\omega^f)({E_1^f})}e_{12}ae_{21}+e_{12}ae_{21}\iota(2\omega^f)(E_1^f)e_{1}\\
&=e_1(\Phi_2^{-1}d\Phi_2+d\Phi_2\Phi_2^{-1})^f\epsilon a \epsilon-\epsilon a\epsilon(\Phi_2^{-1}d\Phi_2+d\Phi_2\Phi_2^{-1})^fe_1\\
&\qquad-e_1(\Phi_1^{-1}d\Phi_1+d\Phi_1\Phi_1^{-1})^fe_{12}ae_{21}+e_{12}ae_{21}(\Phi_1^{-1}d\Phi_1+d\Phi_1\Phi_1^{-1})^fe_{1}\\
&=e_{12}(\Phi_2^{-1}d\Phi_2+d\Phi_2\Phi_2^{-1})e_{21} \epsilon a \epsilon-\epsilon a \epsilon e_{12}(\Phi_2^{-1}d\Phi_2+d\Phi_2\Phi_2^{-1})e_{21}\\
&\qquad-(\Phi_1^{-1}d\Phi_1+d\Phi_1\Phi_1^{-1})e_{12}ae_{21}+e_{12}ae_{21}(\Phi_1^{-1}d\Phi_1+d\Phi_1\Phi_1^{-1}) \\
&= [e_{12} a e_{21},  \Phi^{-1}_1 d \Phi_1+d\Phi_1 \Phi_1^{-1}] -[\epsilon a \epsilon , (\Phi_2^{f})^{-1} d\Phi_2^f+d\Phi_2^f (\Phi_2^{f})^{-1}].
\end{align*}
Putting everything together yields~\eqref{compafusion} as expected.
\end{proof}

\subsection{From Calabi--Yau structures to quasi-bisymplectic structures}

Let again $\mathcal C$ be a $k$-linear category with objects set $I=\{1,\dots,n\}$. Set $e_i=\mathrm{id_i}$, $R=\oplus_{i\in I}ke_i$ and $\mathcal T:= \coprod_{i\in I} k[x^{\pm1}_i]$.

\begin{theorem}\label{thmcyqbs}
Assume that we have a relative $1$-Calabi--Yau structure on a $k$-linear functor $\mu:\mathcal T\to \cC$ which induces the natural $1$-Calabi--Yau structure on each $k[x_i^{\pm1}]$. If $A=A_{\cC}$ is $1$-smooth, then it is quasi-bisymplectic with multiplicative moment map $\sum_{i=1}^n \mu(x_i)$. 
\end{theorem} 

\begin{proof}
Define $\Phi:k[x^{\pm1}]\to A$ by $\Phi(x)=\sum_{i=1}^n \mu(x_i)\in\oplus_{i\in I}e_i A^\ast e_i$. Since $\mu$ is $1$-Calabi--Yau, using the notation of \cref{1cypm}, we know that there exists $\omega_k\in\bar\Omega_R^{2k}A$ for all $k$ such that \[
(\iota_E-ud)\bigg(\sum_{k\ge0}u^k\omega_{k+1}\bigg)=\Phi(\gamma)\]
or equivalently\begin{align*}
\iota_E\omega_1&=\Phi(\gamma_1)=\dfrac{1}{2}(\Phi^{-1}d\Phi+d\Phi\Phi^{-1})&&(\mathbb{B}2)\\
\iota_E\omega_2-d\omega_1&=-\dfrac{1}{6}\Phi(\gamma_2)\Rightarrow d\omega_1=\dfrac{1}{6}(\Phi^{-1}d\Phi)^3\mod [-,-]&&(\mathbb B1)\\
\iota_E\omega_3-d\omega_2&=\dfrac{2!}{5!}\Phi(\gamma_3)\\
\vdots\\
\iota_E\omega_{k+1}-d\omega_{k}&=(-1)^k\dfrac{k!}{(2k+1)!}\Phi(\gamma_{k+1})&&k\ge1.
\end{align*}
For $(\mathbb{B}3)$, set $T=k[x^{\pm1}]$ and
write the relative $1$-pre-Calabi--Yau structure \[
A^\vee[1]\to T^\vee[1]\otimes_{T^e}A^e\overset{\gamma}{\simeq} T\otimes_{T^e}A^e\to A\]
 with short resolutions (thanks to our $1$-smoothness assumption) to get the homotopy commuting diagram\[
\xymatrix{A^e\ar[r]^-{\mathrm{id}}\ar[d]_-{E}&A^e\ar[rrr]^-{(\Phi^{-1}\otimes1+1\otimes \Phi^{-1})/2}\ar[d]&&&A^e\ar[r]^-{d\Phi}\ar[d]&\Omega_{A/R}\ar@{_(->}[d]\\
D_{A/R}\ar[r]_-{\mathrm{ev}_\Phi}&A^e\ar[rrr]_-{(\Phi^{-1}\otimes1+1\otimes \Phi^{-1})/2}&&&A^e\ar[r]_-{\mathrm{id}}&A^e}\]
where the homotopy $D_{A/R}\to\Omega_{A/R}$ gives $\iota_E\omega_1=(\Phi^{-1}d\Phi+d\Phi\Phi^{-1})/2$.

Now assume that our Calabi--Yau structure is non-degenerate, that is \[
A^\vee[1]\simeq\mathrm{hofib}\left( T^\vee[1]\otimes_{T^e}A^e\overset{\gamma}{\simeq} T\otimes_{T^e}A^e\to A\right).\]
In short resolutions, this yields a quasi-isomorphism (between vertical complexes)\[
\xymatrix{A^e\ar[rrr]^-{(\Phi^{-1}\otimes1+1\otimes \Phi^{-1})/2}\ar[d]_-{E}&&&A^e\ar[d]^-{d\Phi}\\
D_{A/R}\ar[rrr]_-{\iota_E\omega_1}&&&\Omega_{A/R}}\]
which in particular gives a surjection $D_{A/R} \to \Omega_{A/R}/\langle d\Phi\rangle$, that is $(\mathbb B_3)$.
\end{proof}

\subsection{Fusion}

Set $\mathcal T_{\ge3}=\amalg_{i\ge3}k[x_i^{\pm1}]$ and consider the following composition of $1$-Calabi--Yau cospans \begin{equation}\label{multfus}
\xymatrix{&&\cC^f&&\\
& k\langle x^{\pm1} ,y^{\pm1}\rangle\amalg\mathcal T_{\ge3}\ar[ur]&&~~~~~\mathcal C~~~~~\ar[ul]&\\
k[z^{\pm1}]\amalg\mathcal T_{\ge3} \ar[ur]&&\mathcal T\ar[ul]\ar[ur]&&~~~~\varnothing~~~~\ar[lu]}\end{equation}
where the leftmost one is induced by the pair-of-pants. 
We want to prove the following multiplicative analog of \cref{fusthmadd}.

\begin{theorem}\label{1cyqbs}
Consider a $1$-Calabi--Yau functor $\mathcal T\rightarrow\mathcal C$ inducing the natural $1$-Calabi--Yau structure on each $k[x_i^{\pm1}]$, and assume that $A_\mathcal C$ is $1$-smooth.
Then the quasi-bisymplectic structure on $\cC^f$ induced thanks to \cref{thmcyqbs} by the $1$-Calabi--Yau functor\[
k[z^{\pm1}]\amalg\mathcal T_{\ge3}\rightarrow\cC^f\]
is the  one obtained by fusion of $1$ and $2$ from the quasi-bisymplectic structure of $A_\cC$ induced by \cref{thmcyqbs}.
\end{theorem}

\begin{proof}
Denote by $\Phi_1^f,\Phi_2^f$ the images of $x=x_1,y=x_2$ in the pushout $\cC^f$.
The extra difficulty here with respect to the proof of \cref{fusthmadd} is that the homotopy $\beta_1$ involved in the pair-of-pants cospan is nontrivial, see \cref{popex}.
This non-degenerate homotopy\[
\beta_1=\dfrac{1}{2}\Big(y^{-1}\otimes x^{-1}\otimes xy-y\otimes y^{-1}x^{-1}\otimes x\Big)\in\overline{\mathrm{HH}}_2k\langle x^{\pm1} ,y^{\pm1}\rangle\]
 is mapped in $\overline{\mathrm{DR}}^2 k\langle x^{\pm1} ,y^{\pm1}\rangle$ to \begin{align*}
\omega&=\dfrac{1}{4}\Big(y^{-1}d x^{-1}d (xy)-yd( y^{-1}x^{-1})d x\Big)\\
&=\dfrac{1}{4}\Big(-y^{-1}x^{-1}d xx^{-1}(xd y+dxy)+d yy^{-1}x^{-1}d x+x^{-1}dxx^{-1}dx\Big)\\
&=\dfrac{1}{4}\Big(-y^{-1}x^{-1}d xd y-y^{-1}x^{-1}d xx^{-1}dxy+d yy^{-1}x^{-1}d x+x^{-1}dxx^{-1}dx\Big)\\
&\equiv-\dfrac{1}{2}x^{-1}dxdyy^{-1}\quad\mod[-,-]
\end{align*}
which is mapped to\[
-\dfrac{1}{2}(\Phi_1^f)^{-1}d\Phi_1^fd\Phi^f_2(\Phi_2^f)^{-1}\in \overline{\mathrm{DR}}^2_{R^f} \cC^f.\]
The \cref{fusionqH} allows us to conclude, thanks to the uniqueness~\cite[Theorem 7.1]{VdB2} of compatibility and~\cite[Theorem 8.2.1]{VdB2}.
\end{proof}

To summarize, we have proven that the following diagram commutes, where $R^f=\oplus_{i\in I\setminus\{2\}}ke_i$ and ${\mathcal T}^f=\amalg_{i\in I\setminus\{2\}}k[x_i^{\pm1}]$.\[
\xymatrix{
{\left\{\begin{array}{c}1$-Calabi--Yau functors$ \\ {\mathcal T}\to\cC$, over $\mathcal{R},\\
$with $A_\cC$ $1$-smooth$\end{array}\right\}}\ar[rrr]^-{\text{Theorem }\ref{thmcyqbs}}\ar[dd]_-{\substack{\text{composition}\\\text{with multiplicative}\\\text{pair-of-pants}}}&&&{\left\{\begin{array}{c}$quasi-bisymplectic structures$\\$on $1$-smooth $R$-algebras$\end{array}\right\}}\ar[dd]^-{\text{fusion}}\\&&&\\
{\left\{\begin{array}{c}1$-Calabi--Yau functors$ \\ {\mathcal T}^f\to\cC^f$, over $\mathcal{R}^f,\\
$with $A_{\cC^f}$ $1$-smooth$\end{array}\right\}}\ar[rrr]^-{\text{Theorem }\ref{thmcyqbs}}&&&{\left\{\begin{array}{c}$quasi-bisymplectic structures$\\$on $1$-smooth $R^f$-algebras$\end{array}\right\}}
}\]

\subsection{Examples}\label{subsecexqq}

\subsubsection{An elementary quiver}\label{a2ex}

Consider the quiver $A_2=(V=\{1,2\},E=\{e:1\to2\})$, with orthogonal idempotents $e_1$ and $e_2$ satisfying $1=e_1+e_2$, $R=ke_1\oplus ke_2$, and set \[a_1=e_1+e^*e\text{ and }a_2=e_2+ee^*.\]
Let us denote by $A$ the localization $(k\overline{A_2})_{a_{1},a_2}$. Recall that we have given in \cite{BCS2} a relative $1$-Calabi--Yau structure on $ \Phi: k[x^{\pm1}] \to A$  defined by \[
\Phi_1(x_1)=a_1^{-1}\quad \text{and}\quad \Phi_2(x_2)=a_2.
\]
Define $\partial/\partial e$ and $\partial/\partial e^*$ in $D_RA$ by $\partial e/\partial e=e_2\otimes e_1$, $\partial e^*/\partial e=0$, $\partial e^*/\partial e^*=e_1\otimes e_2$ and $\partial e/\partial e^*=0$.

In the previous section we proved that this Calabi--Yau structure induces a quasi-bisymplectic one $\omega_1\in\overline{\mathrm{DR}}_R^2A$ on $A$. We want to prove the following.

\begin{proposition} The double quasi-Poisson bracket compatible with $\omega_1$ through~\cite[Theorem 7.1]{VdB2} is the one described in~\cite[\S 8.3]{VdB2}:
\[P=\dfrac{1}{2} \left(\left(1+ee^* \right) \dfrac{\partial}{\partial e^*} \frac{ \partial}{\partial e} -\left(1+e^*e\right) \dfrac{\partial}{\partial e}\frac{ \partial}{\partial e^*}\right) \in \left( D_RA/[D_RA, D_RA] \right)_2.\] 
\end{proposition}

Note that we use the convention regarding concatenation of paths opposite to the one in~\cite{VdB}, that is $e=e_2ee_1$.

\begin{proof}
In \cite{BCS}, one homotopy $\phi(\gamma_1) \sim 0$ is given by
\begin{align}\label{homotopyA2}\begin{split}
\beta_1&=\dfrac{1}{2}\big(e^*\otimes e\otimes \Phi+\Phi\otimes e^*\otimes e- e^*\otimes \Phi^{-1}\otimes e-\Phi^{-1}\otimes e\otimes e^*\\
&\qquad\qquad\qquad\qquad\qquad\qquad\qquad\qquad+1\otimes e^*\otimes e \Phi-1\otimes e\Phi\otimes e^*\big)\end{split}\end{align}
where $\Phi=\Phi_1(x_1)+\Phi_2(x_2)$. It yields an element ($1/4$ appears because of the degree operator)\[
\omega_1=\dfrac{1}{4}\big(e^*ded\Phi+\Phi de^*de-e^*d\Phi^{-1}de-\Phi^{-1}dede^*+de^*d(e\Phi)-d(e\Phi)de^*\big)\]
in $\overline{\mathrm{DR}}^2A= \left( \overline\Omega A/[\overline\Omega A,\overline \Omega A] \right)_2$.
We can heavily simplify this expression working modulo $[\overline\Omega A,\overline \Omega A]$. First note that (again, $dab$ stands for $(da)b$)
\begin{align*}
d\Phi&=-a_1^{-1}(de^*e+e^*de)a_1^{-1}+dee^*+ede^*=-\Phi(de^*e+e^*de)\Phi+dee^*+ede^*\\
d\Phi^{-1}&=de^*e+e^*de-a_2^{-1}(dee^*+ede^*)a_2^{-1}=de^*e+e^*de-\Phi^{-1}(dee^*+ede^*)\Phi^{-1},
\end{align*}
thus, using $\Phi e\Phi=e$ and $\Phi e^*\Phi=e^*$ (cf~\cite[(4.3)]{BCS2}),\begin{align*}
4\omega_1&=\Phi de^*de-\Phi^{-1}dede^*+e^*ded\Phi-e^*d\Phi^{-1}de+2de^*d(e\Phi)\\
&=\Phi de^*de-\Phi^{-1}dede^*-e^*de\Phi(de^*e+e^*de)\Phi\\
&\qquad+e^*\Phi^{-1}(dee^*+ede^*)\Phi^{-1}de+2de^*de\Phi-2de^*e\Phi(de^*e+e^*de)\Phi\\
&=\Phi de^*de-\Phi^{-1}dede^*-e^*de\Phi de^*e\Phi\\
&\qquad\underbrace{-e^*de\Phi e^*de\Phi+e^*\Phi^{-1} dee^*\Phi^{-1}de}_{\equiv0}+e^*\Phi^{-1}ede^*\Phi^{-1}de\\
&\qquad\qquad+2de^*de\Phi-2\underbrace{de^*e\Phi de^*e\Phi}_{\equiv0}-2de^*e\Phi e^*de\Phi\\
&\equiv3\Phi de^*de-\Phi^{-1}dede^*-e\Phi e^*de\Phi de^*+e^*\Phi^{-1}ede^*\Phi^{-1}de+2de^*e\Phi e^*de\Phi\\
&=3\Phi de^*de-\Phi^{-1}dede^*-ee^*\Phi_2^{-1}de\Phi de^*+e^*e\Phi_1 de^*\Phi^{-1}de+2de^*ee^*\Phi_2^{-1} de\Phi\\
&=3\Phi de^*de-\Phi^{-1}dede^*-de\Phi de^*+\Phi^{-1}de\Phi de^*\\
&\qquad+ de^*\Phi^{-1}de-\Phi de^*\Phi^{-1}de-2de^*de\Phi+2de^*\Phi^{-1}de\Phi\\
&\equiv2\Phi de^*de-2\Phi^{-1}dede^*.
\end{align*}
We now need to prove that $P$ and $\omega_1$ are compatible, meaning as previously that\begin{equation}\label{compT}
\iota(\omega_1)\iota(P)=1-\dfrac{1}{4}T\end{equation}
with $T(dp)=[p,\Phi^{-1}d\Phi-d\Phi\Phi^{-1}]$. For $p=e$, the LHS is\begin{align*}
\iota(\omega_1)\iota(P)(de)&=\dfrac{1}{2}\iota(\omega_1)\left(\dfrac{\partial}{\partial e^*}(1+e^*e)+(1+ee^*)\dfrac{\partial}{\partial e^*}\right)\\
&=\dfrac{1}{2}({}^\circ i_{{\partial}/{\partial e^*}}(\omega_1)(1+e^*e)+(1+ee^*){}^\circ i_{{\partial}/{\partial e^*}}(\omega_1))
\end{align*}
where\[
i_{\delta}(pdqdr)=p\delta(q)'\otimes\delta(q)''dr-pdq\delta(r)'\otimes\delta(r)''\in A\otimes\Omega^1+\Omega^1\otimes A\]
as stated earlier.
Note that above we have used, for $\pi,\nu\in A$ and $\delta\in D_{A/R}$,\begin{align*}
{}^\circ i_{\pi\delta\nu}(pdqdr)&={}^\circ(p\delta(q)'\nu\otimes\pi\delta(q)''dr-pdq\delta(r)'\nu\otimes\pi\delta(r)'')\\
&=\pi{}^\circ i_{\delta}(pdqdr)\nu
\end{align*}
since the bimodule structure on $D_{A/R}$ is induced by the inner one on $A^e$, as explained in the proof of~\cite[2.8.6]{CBEG}. We have:\begin{align*}
{}^\circ i_{{\partial}/{\partial e^*}}(2\omega_1)={}^\circ(\Phi\otimes de+\Phi^{-1}de\otimes e_2)=de\Phi+\Phi^{-1}de\end{align*}
thus\begin{align*}
4\iota(\omega_1)\iota(P)(da)&=(de\Phi+\Phi^{-1}de)(1+e^*e)+(1+ee^*)(de\Phi+\Phi^{-1}de)\\
&=2de+\Phi^{-1}de\Phi^{-1}+\Phi de\Phi
\end{align*}
whereas $4$ times the RHS of~\eqref{compT} evaluated at $de$ is\begin{align*}
4de-[e,\Phi^{-1}d\Phi-d\Phi\Phi^{-1}]&=4de-e\Phi^{-1}(-\Phi(de^*e+e^*de)\Phi+dee^*+ede^*)\\
&\qquad+e(-\Phi(de^*e+e^*de)\Phi+dee^*+ede^*)\Phi^{-1}\\
&\qquad\qquad+\Phi^{-1}(-\Phi(de^*e+e^*de)\Phi+dee^*+ede^*)e\\
&\qquad\qquad\qquad-(-\Phi(de^*e+e^*de)\Phi+dee^*+ede^*)\Phi^{-1}e\\
&=4de+ede^*e\Phi+ee^*de\Phi-e\Phi de^*e-e\Phi e^*de\\
&\qquad+\Phi^{-1}dee^*e+\Phi^{-1}ede^*e-dee^*\Phi^{-1}e-ede^*\Phi^{-1}e\\
&=4de+ee^*de\Phi-\Phi^{-1}e e^*de+\Phi^{-1}dee^*e-dee^*e\Phi\\
&=4de+\Phi de\Phi-de\Phi-de+\Phi^{-1}de\\
&\qquad+\Phi^{-1}de\Phi^{-1}-\Phi^{-1}de-de+de\Phi\\
&=2de+\Phi^{-1}de\Phi^{-1}+\Phi de\Phi
\end{align*}
as wished. Computations are similar to prove \cref{compT} evaluated at $de^*$.
\end{proof}

\subsubsection{Arbitrary quivers}\label{arbex}

Let us go back to the proof~\cite[Theorem 4.8]{BCS2} of the $1$-Calabi--Yau structure on the multiplicative moment map
$\mu_Q:\coprod_{v\in V}k[z_v^{\pm1}]  \rightarrow  k\overline {Q}_{loc} :=k\overline {Q}[(1+ee^*)^{-1}]_{e\in\overline E}$ defined by \[
z_v  \longmapsto  \prod_{e\in E\cap t^{-1}(v)}(1+ee^*)\times \prod_{e\in E\cap s^{-1}(v)}(1+e^*e)^{-1}.\]
It is done by realizing this functor as successive compositions of Calabi--Yau cospans.
Let us specify an order that better suits our purpose. As usually we denote by $Q^\mathrm{sep}$ the quiver with same edge set $E$ but vertex set $\overline E=\{v_e=s(e),v_{e^*}=t(e)\}$. It is the disjoint union of $|E|$ copies of $A_2$. We have a 1-Calabi--Yau morphism\begin{equation}\label{CYsep}
\mu_{Q^\mathrm{sep}}:\coprod_{e\in E}(k[x_e^{\pm1}]\amalg k[y_e^{\pm1}])\longrightarrow k\overline {Q^\mathrm{sep}}_{loc}
\end{equation}
given by $x_e\mapsto (e_{s(e)}+e^*e)^{-1}$ and $y_e\mapsto e_{t(e)}+ee^*$.
We know thanks to the previous section that the quasi-bisymplectic structure on $k\overline {Q^\mathrm{sep}}_{loc}$ induced by this $1$-Calabi--Yau multiplicative moment map matches the one described by Van den Bergh in~\cite{VdB2}.

We want to prove the same for $Q$ by fusing pairs of vertices $(v_e,v_f)$ any time $s(e)=s(f)$ in $\overline Q$.
Precisely, pick a finite sequence of fusion of pairs of vertices that takes us from $Q^\mathrm{sep}$ to $Q$, and consider an intermediary step $Q^\diamond$. Assume that the quasi-bisymplectic structure induced by the $1$-Calabi--Yau one on $\mu_{Q^\diamond}$ matches Van den Bergh's, and proceed to the next fusion in our sequence. Assume that we fuse $1$ and $2$ in the vertex set $I$ of $Q^\diamond$. We mean by that that we precisely proceed to the composition~\eqref{multfus}, where $\cC=k\overline{Q^\diamond}_{loc}$.
By induction and using \cref{1cyqbs} we get the following.

\begin{theorem}
The quasi-bisymplectic structure on $k\overline{Q}_{loc}$ induced by the $1$-Calabi--Yau one on $\mu_Q$ matches the one given by Van den Bergh.
\end{theorem}


\section{Representation spaces}

As before assume that $A$ is a $1$-smooth $R$-algebra with $R=\oplus_{i\in I}ke_i$ where the $e_i$ are pairwise orthogonal idempotents and $I:=\{1, \cdots, n\}$. For any $I$-graded finite dimensional space $V$ define $A_V$  by\[
\Hom_{\mathrm{Alg}/R}(A,\mathrm{End}(V))=\Hom_{\mathrm{CommAlg}/k}(A_V,k).
\]
Thanks to~\cite[(6.2.2)]{CBEG}, setting $X_V=\mathrm{Spec}(A_V)$, we have a map\begin{equation}\label{deftra}
\underline{\mathrm{tr}}:\mathrm{DR}^* A\longrightarrow \Omega^*( X_V)^{\mathrm{GL}_V}\end{equation}
given by $\alpha\mapsto\mathrm{tr}(\hat\alpha)$ where $\hat\alpha$ is induced by the evaluation\[
A\rightarrow(A_V\otimes \mathrm{End}(V))^{\mathrm{GL}_V}\quad;\quad a\mapsto\hat a.\]

Thanks to~\cite[Proposition 6.1]{VdB2}, there is a quasi-Hamiltonian structure on $(X_V,\underline{\mathrm{tr}}(\omega),\hat\Phi)$ when $(A,\omega,\Phi)$ is quasi-bisymplectic. Now $\hat\Phi:X_V\to \mathrm{GL}_V$ induces a lagrangian structure on $[X_V/\mathrm{GL}_V]\to[\mathrm{GL}_V/\mathrm{GL}_V]$. 

On the other hand, thanks to~\cite{BD2}, if $\Phi$ carries a $1$-Calabi--Yau structure, it yields a lagrangian structure on $\mathrm{Perf}_A\to\mathrm{Perf}_{k[x^{\pm1}]}$, and thus considering substacks on $[X_V/\mathrm{GL}_V]\to[\mathrm{GL}_V/\mathrm{GL}_V]$ again.

In both cases, we know that the induced $1$-shifted symplectic structure on $[\mathrm{GL}_V/\mathrm{GL}_V]$ is the standard one, thanks to~\cite[§5.1]{BCS2} for the latter. 

Now assume that the $1$-Calabi--Yau structure on $\Phi$ induces the quasi-bisymplectic structure $(A,\omega,\Phi)$, that is $\omega_1$ in the proof of \cref{thmcyqbs} is $\omega$.
The current section is devoted to the proof of the following.

\begin{theorem}\label{complagthm}
These two lagrangian structures are identical.
\end{theorem}

\subsection{Lagrangian morphisms and quasi-hamiltonian spaces}

Let $X$ be a smooth algebraic variety. Since we will apply the following results to $X=X_V$ we assume $X$ to be affine for simplicity but these results 
can be extended to the non-affine case. Assume that a reductive group $G$ acts on $X$ and consider a $G$-equivariant morphism $\mu:X\to{G}$, 
which induces $[\mu]:[X/G]\to[G/G]$. Consider the standard 1-shifted symplectic structure on $[G/G]$ given by 
$\underline\omega=\underline\omega_0+\underline\omega_1$ where $\underline\omega_0\in(\Omega^1(G)\otimes \mathfrak g^*)^G$ and 
$\underline\omega_1\in\Omega^3(G)^G$.

We refer to~\cite[§3]{BCS} for a precise definition of the space $\mathcal A^{p,(\mathrm{cl})}(X,n)$ of (closed) $p$-forms of degree $n$ on $X$.
When $\alpha \in \Omega^2(X)^G$, we say that $(\alpha,\mu)$ satisfy the multiplicative moment condition if \begin{equation}\label{multmo}
\forall u\in\mathfrak g,~~i_{\vec u}\alpha=\langle\mu^*\underline\omega_0,u\rangle.\tag{$\mathbb M$}\end{equation}
This is condition (B2) in~\cite{VdB2}.

\begin{lemma}
The space of homotopies between $[\mu]^*\underline\omega_0$ and $0$ in $\mathcal A^{2,\mathrm{cl}}([X/G],1)$ is discrete. It is the space of invariant 
$2$-forms $\alpha\in\Omega^2(X)^G$ satisfying~\eqref{multmo}.
\end{lemma}

\begin{proof}
The \textit{cochain} complex of $2$-forms on $[X/G]$ is given by 
\[
\xymatrix{
\Omega^2(X)^G\ar[r]^-\partial&(\Omega^1(X)\otimes \mathfrak g^*)^G\ar[r]&(\mathcal O(X)\otimes S^2 \mathfrak g^*)^G}.
\]
The result follows from the fact that, by definition, $\partial$ is given by $\langle\partial\alpha,u\rangle=i_{\vec u}\alpha$ for every $u\in\mathfrak g$.
\end{proof}

This can be extended to the following, where we recognize the extra condition (B1) of~\cite{VdB2}.

\begin{lemma}
The space of homotopies between $[\mu]^*\underline\omega$ and $0$ in $\mathcal A^{2,\mathrm{cl}}([X/G],1)$ is discrete. It is the space of $2$-forms 
$\alpha\in\Omega^2(X)^G$ satisfying~\eqref{multmo} and
\begin{align*}
d_{\mathrm{dR}}\alpha&=\mu^*\underline\omega_1.
\end{align*}
\end{lemma}

\begin{proof}
The de Rham (cochain) complex of $[X/G]$ in weight $\ge2$ is the total (cochain) complex of the bicomplex
\[
\xymatrix{
&&&\\
\Omega^3(X)^G\ar@{--}[u]\ar[r]&(\Omega^2(X)\otimes \mathfrak g^*)^G\ar@{--}[u]\ar[r]&(\Omega^2(X)\otimes  S^2\mathfrak g^*)^G\ar@{--}[u]\ar[r]&(\mathcal O(X)\otimes S^3 \mathfrak g^*)^G\ar@{--}[u]\\
\Omega^2(X)^G\ar[u]^-{d_{\mathrm{dR}}}\ar[r]^-\partial&(\Omega^1(X)\otimes \mathfrak g^*)^G\ar[u]\ar[r]&(\mathcal O(X)\otimes S^2 \mathfrak g^*)^G\ar[u].&
}
\]
The space of 2-forms $\alpha\in\Omega^2(X)^G$ mapped on $\mu^*\omega\in\Omega^3(X)^G\oplus (\Omega^1(X)\otimes \mathfrak g^*)^G$ 
by $d_{\mathrm{dR}}\oplus\partial$ has the expected description.
\end{proof}

Now thanks to~\cite{PTVV}, the non-degeneracy condition (that is (B3) in~\cite{VdB2}) defines an union of connected components in the space of (closed) $2$-forms.
Therefore we have the following result (which is already implicit in \cite{Cal,Saf}). 
\begin{theorem}
The space of lagrangian structures on $[\mu]$ is discrete; it is the set of $2$-forms $\alpha\in\Omega^2(X)^G$ such that~\eqref{multmo}.
\end{theorem}
In particular, the space of lagrangian structures on $[\mu]$ (or, equivalently, the set of quasi-hamiltonian structures on $X$ with group valued moment map $\mu$) 
is a subset of $\Omega^2(X)$. 
\begin{corollary}
Two lagrangian structures on $[\mu]$ coincide if and only if the associated $2$-forms on $X$ are the same. 
\end{corollary}

\begin{remark}
Here is how we understand geometrically the $2$-form on $X$ we get from an $\alpha$ satisfying~\eqref{multmo}. The pull-back of $\underline\omega_0$ 
along the quotient $G\to[G/G]$ is zero. As $[\mu]^*\underline\omega_0\sim0$ via $\alpha$, we get a self-homotopy of $0$ in the space $2$-forms of degree 
$1$ on the fiber product
\[
[X/G]\underset{[G/G]}{\times}G\simeq X.
\]
Such a self-homotopy is a $2$-form of degree $0$ on $X$, which is nothing but $\alpha$.
\end{remark}

\subsection{Identifying two lagrangian structures: proof of \cref{complagthm}}

Consider the composition \[
\mathrm{Spec}(A_V)=X_V\twoheadrightarrow[X_V/\mathrm{GL}_V]\hookrightarrow\mathrm{Perf}_A.\]
 It is given by an $A-A_V$-bimodule $M$ which induces a chain\[
\xymatrix{
\overline{\mathrm{HH}}A\ar@/_1.7pc/[rr]_-{a\mapsto \hat a}\ar[r]&\overline{\mathrm{HH}}(\mathrm{Mod}_{A_V}^\text{perf})&\overline{\mathrm{HH}}(\mathrm{End}_{A_V}(M))\ar[r]^-{\mathrm{tr}}\ar[l]_-{\sim}&\overline{\mathrm{HH}}A_V\simeq\Omega^* A_V}
\]
given by\[
a_0\otimes a_1\otimes\dots\otimes a_n\mapsto\mathrm{tr}(\hat a_0)d\mathrm{tr}(\hat a_1)\dots d\mathrm{tr}(\hat a_n),\]
that is $\underline{\mathrm{tr}}$ again, \textit{cf}~\eqref{deftra}. Thus the $2$-forms match on $X_V$, and therefore the associated lagrangian structures as well thanks to the previous subsection.

\begin{example}
\begin{enumerate}
\item[(i)]
Let us get back to \cref{a2ex}, where $A$ is a localization of the path algebra of the $A_2$ quiver and $\Phi$ denotes the associated multiplicative moment map. Thanks to the computations in \cref{a2ex}, \cref{complagthm} applies and the $1$-Calabi--Yau structure on $\Phi$ exhibited in~\cite{BCS2} induces the same lagrangian structure on 
\[
\big[\hat\Phi\big]:\big[\mathrm{Rep}(A,\vec n)/GL_{\vec n}\big]{\longrightarrow}\big[GL_{\vec n}/GL_{\vec n}\big],
\]
for some dimension vector $\vec n=(n_1,n_2)$, as the one induced by Van den Bergh's quasi-Hamiltonian $GL_{\vec n}$-structure in~\cite{VdB2}.
\item[(ii)] Similarly, using \cref{arbex}, we finally prove the conjecture raised in~\cite[\S5.3]{BCS2} which is the identical statement for an arbitrary quiver $Q$.
\end{enumerate}
\end{example}

\end{document}